\documentclass[a4paper,10pt]{article}
\usepackage{tikz}
\usepackage[utf8]{inputenc}
\usepackage{amsthm}
\usepackage{color}
\newtheorem{remark}{Remark}
\newtheorem{proposition}{Proposition}[section]
\newtheorem{corollary}{Corollary}[section]
\newtheorem{lemma}{Lemma}[section]
\newcommand{\mycomment}[1]{ }

\def\Oij{\bO_{i,j}}

\usepackage{amssymb}
\usepackage{amsfonts}
\usepackage{amsmath}

\usepackage{amsthm}

\newcommand\beq{\begin{equation}}
\newcommand\eeq{\end{equation}}

\usepackage{dsfont}

\renewcommand{\to}{\rightarrow}

\renewcommand{\det}{\operatorname{det}}
\renewcommand{\div}{\operatorname{div}}

\newcommand{\tr}{\operatorname{tr}}

\newcommand{\grad}{\boldsymbol{\nabla}}

\newcommand{\bA}{\boldsymbol{A}} 

\newcommand{\bB}{\boldsymbol{B}}

\def\bD{\boldsymbol{D}}

\newcommand{\be}{\boldsymbol{e}}

\newcommand{\bF}{\boldsymbol{F}}
\newcommand{\bg}{\boldsymbol{g}}
\newcommand{\bG}{\boldsymbol{G}}
\newcommand{\bH}{\boldsymbol{H}}

\newcommand{\bI}{\boldsymbol{I}} 

\newcommand{\bL}{\boldsymbol{L}}

\def\bn{\boldsymbol n}

\newcommand{\bO}{\boldsymbol{O}}

\def\bR{\boldsymbol R}

\newcommand{\bu}{\boldsymbol{u}}
\newcommand{\bU}{\boldsymbol{U}}

\newcommand{\bW}{\boldsymbol{W}}

\newcommand{\bx}{\boldsymbol{x}} 

\newcommand{\bY}{\boldsymbol{Y}}

\newcommand{\bSigma}{\boldsymbol{\Sigma}}

\def\Acal{\mathcal{A}}

\def\Ecal{\mathcal{E}} 			

\def\Gcal{\mathcal{G}}

\def\Pcal{\mathcal{P}}

\def\Vcal{\mathcal{V}}

\def\Zcal{\mathcal{Z}}

\newcommand{\EE}{\mathbb{E}}

\newcommand{\N}{\mathbb{N}}
\newcommand{\NN}{\mathbb{N}}

 \newcommand{\R}{\mathbb{R}} 

\newcommand{\bzero}{{\boldsymbol{0}}}

\title{ 
 Viscoelastic flows with conservation laws \\
 {\small 
 The shallow-water regime for Maxwell fluids}
}
\author{
S\'ebastien Boyaval, Laboratoire d'hydraulique Saint-Venant \\
Ecole des Ponts 
-- EDF R\&D -- CEREMA \\ 
EDF'lab 6 quai Watier 78401 Chatou Cedex France, \\
\& MATHERIALS, INRIA Paris (sebastien.boyaval@enpc.fr)
}
\newcommand{\ro}{H}
\newcommand{\ux}{U}
\newcommand{\uy}{V}

\newcommand{\sx}{B_{xx}}
\newcommand{\sy}{B_{yy}}
\newcommand{\sxy}{B_{xy}}
\newcommand{\sz}{B_{zz}}

\begin{document}

\maketitle

\begin{abstract}

We propose in this work the first 
symmetric hyperbolic 
system of \emph{conservation laws} to describe viscoelastic flows 
of  
Maxwell fluids, i.e. fluids with memory that are characterized by one 
relaxation-time parameter.

Precisely, the system of 
quasilinear PDEs 
is detailed 
for the 
\emph{shallow-water} regime,
i.e. for 
hydrostatic incompressible 2D flows with free surface under gravity.
It generalizes Saint-Venant system 
to \emph{viscoelastic} flows of Maxwell fluids, 
and encompasses previous 
works with F.~Bouchut.
It also generalizes the 
(thin-layer) elastodynamics 
of hyperelastic materials to viscous fluids,
and 
to various rheologies between solid and liquid states 
that can be formulated using our new 
variable as material parameter. 

The new viscoelastic flow model has many potential 
applications, additionally to 
falling into the 
theoretical framework 
of (symmetric hyperbolic) 
systems of conservation laws.
In computational rheology, it offers a new approach to the High-Weissenberg Number Problem (HWNP).
For transient geophysical flows, it offers perspectives of thermodynamically-compatible numerical simulations,
with a Finite-Volume (FV) discretization say.
Besides, one FV discretization of the new continuum model is proposed herein to precise our ideas 
incl. the physical meaning of the solutions.
Perspectives are finally listed after some 
numerical simulations.

\end{abstract}

\section{Introduction} 
\label{sec:intro}

Many mathematical models have been proposed for \emph{flows of real 
fluids}. 
Like the celebrated Navier-Stokes equations \cite{lions-1996}, 
they mainly account for \emph{viscosity} as a manifestation of the fluid non-ideality.
But other macroscopic manifestations of the fluid 
microstructure in continuum mechanics can also be accounted for, 
typically through complex deviatoric stresses \cite{renardy-2000}. 

However, most of those models have a diffusion form 
; thus they violate the physical principle of a \emph{finite} speed of propagation for material/energetic perturbations.
This violation may not be so important in some applications 
; but it seems important to us e.g. in geophysics, 
when one is interested in the propagation of perturbations though \emph{large} systems. 

Moreover, a diffusion form also often entails unnecessary energy dissipation in numerical simulations ;
this prevents one from capturing realistic transient dynamics on \emph{long} time ranges.


\smallskip

A few \emph{hyperbolic} systems of quasilinear PDEs have been proposed to model fluid flows 
with realistic shear-waves propagating at finite-speed.
They rely on seminal ideas of Maxwell \cite{Maxwell01011866,Maxwell01011867,maxwell-1874} after Poisson \cite{poisson-1831}
to produce viscous (frictional/dissipative) effects through the relaxation of an elastic deformation 
(i.e. a source term produces entropy in shear flows). 

For instance, \cite{PHELAN1989197,EDWARDS1990411} proposes a model for slightly compressible 2D flows of 
viscoelastic fluids, see also \cite{Guaily2010158} for recent numerical simulations. 
Although the latter model does not preserve mass, it is quite close to a model that we proposed in \cite{bouchut-boyaval-2015} 
and where mass conservation can be ensured, see \cite{boyaval-hal-01661269}. 
But 
the latter model has a non-conservative form (see \cite{boyaval-hal-01661269})
with genuinely nonlinear products that seem difficult to interpret 
for multi-dimensional flows.
Another very interesting model has also been proposed recently in \cite{Peshkov2016,peshkov2018arXiv180600706P},
which aims at general 
flows of viscous fluids.
Like in new our model, that very inspiring work introduces an additional ``state variable'' for viscoelastic 
matter.
But the latter very recent proposition still deserves discussing with respect to its performance in applications, to its generalization to various 
rheologies and to its 
justification.
Here, we propose an alternative 
similar in spirit, but mathematically different.

\smallskip

In the present work, we propose a multi-dimensional symmetric-hyperbolic system of conservation laws
to model flows of (compressible) viscoelastic fluids of Maxwell type, which we believe new and 
useful for any flow and for various rheologies. 
The new flow model of Maxwell fluids is 
detailed 
for \emph{2D isothermal flows of incompressible fluids with a free-surface and a hydrostatic pressure}, see eq.~\eqref{eq:SVUCM0} in Section~\ref{sec:model}.
We have detailed our new viscoelastic approach in a low-dimensional case here,
i.e. an analog of 2D Euler equations 
proposed e.g. by Saint-Venant \cite{saint-venant-1871} for shallow-water flows, 
because 
(i) our ideas are easier to understand theoretically and numerically then than in the most general case, 
(ii) that framework is useful already to a number of 
applications in geophysics (see below), and
(iii) 
it is sufficient to see that our model 
generalizes to 3D flows and more complex fluids. 

The new viscoelastic model contains our previous 1D model \cite{bouchut-boyaval-2013} as a closed subsystem for translation-invariant solutions,
and our previous 2D models \cite{bouchut-boyaval-2015,boyaval-hal-01661269} without a conservative formulation, see eq.~\eqref{eq:SVUCM} below.
In particular, it contains the standard shallow-water system of Saint-Venant in the zero elasticity limit $G\to 0$.
To encompass the limitations of our previous models, 
we have 
interpreted them as particular ``closed'' hyperbolic subsystems of the new model.
We have generalized the elastodynamics system of hyperelastic materials to Maxwell fluids thanks to a new state variable $\bA$ to that aim, 
and our new model thus also contains standard elastodynamics equations when $\bA$ is uniform in space and time 
(i.e. at large \emph{Deborah, or Weissenberg number} $\lambda\gg1$).

\smallskip

%
Our new model has a molecular (``microscopic'') justification, formally at least,
which is reminiscent of the molecular theory of polymer solutions \cite{ottinger-1996-Book,doi-edwards-1998}, see Rem.~\ref{rem:kinetic}.
The new state variable is one kind of ``order parameter'' which accounts for the flow-induced distortion of the microsctructure,
relaxes due to thermal agitation at non-zero temperature 
and then 
creates viscous effects.
This is thermodynamically compatible : one can formulate the second principle 
with an entropy functional, 
or with a Helmholtz free-energy that decays following a Clausius-Duhem inequality at non-zero temperature.

\smallskip

Our model copes well with a general formalism 
desired for the mathematical modelling of continuum mechanics according to the admitted physical principles \cite{muller-ruggeri-1998}.
As a consequence, our model straightforwardly benefits from the 
numerous efforts toward a precise mathematical understanding of (solutions to the Cauchy problem for) symmetric-hyperbolic systems of conservation laws \cite{dafermos-2000},
and toward accurate numerical simulations of physically-meaningful solutions after discretization e.g. by the Finite-Volume (FV) method \cite{godlewski-raviart-1996,leveque-2002}.
In particular, our model is close to well-known systems:
the Euler gas dynamics system (formally equivalent to Saint-Venant shallow-water system when 2D and barotropic)
and the 
elastodynamics system for hyperelastic materials \cite{wagner-2009}.

Of course, our model is also limited by the present state of the mathematical theory 
(multi-dimensional 
solutions are not 
well defined yet, see \cite{chiodaroli-delellis-kreml-2015} for instance), 
and of the FV method.
For ``entropy-stability'', one would like the free-energy to decay at the discrete level;
however it is not globally convex with respect to the whole set of conservative variables (unlike the energy).
Moreover, involutions like in the 
elastodynamics system for hyperelastic materials are important 
but hardly preserved discretely \cite{wagner-2009,kluth-despres-2010}.

However, systems of conservation laws are still 
much studied theoretically \cite{Demoulini2001,demoulini-stuart-tzavaras-2012,christoforou-galanopoulou-tzavaras-2018}
and numerically \cite{despres-2017}. So we believe our model has good perspectives.
And even if at present, one has to pragmatically resort to heuristics and empiricism in 
numerical simulations,
our model can already have many applications.
%
For instance, it could be useful in environmental hydraulics to model complex fluid flows (turbulent/non-Newtonian)
with generalized Saint-Venant equations.
%
See in section~\ref{sec:saintvenant}
the modelling issue for \emph{real fluids} in Saint-Venant 2D framework. 

\mycomment{To that aim however, one may want to extend our model to non-flat bottom. We leave this technicality to further work.}

In section~\ref{sec:maxwell}, our new Saint-Venant-Maxwell (SVM) model 
is introduced starting from the more standard SVUCM model for viscoelastic flows of Upper-Convected Maxwell fluids. 
We hope that our new viscoelastic approach solves a number of difficulties in computational rheology
like the High-Weissenberg Numebr Problem (HWNP) \cite{owens-philips-2002},
and generalization to compressible models \cite{Bollada2012} for thermodynamically-compatible mass transfers.
In particular, we recall that the standard viscoelastic models 
usually require strict incompressibility 
with a non-zero retardation time (a background viscosity) \cite{keunings-1989,owens-philips-2002}. 
Now, whereas 
the standard approach can be useful numerically in computational rheology for small-data solutions, to avoid 
singularities \cite{speziale-2000}, the trick 
does not seem to work well at high Weissenberg numbers \cite{keunings-1989,owens-philips-2002}; 
moreover, it naturally limits the scope of application. 

\smallskip

In Section~\ref{sec:FV} we propose a Finite-Volume (FV) discretization of our new model,
using a 1D Riemann solver of relaxation type (``\`a la Suliciu'') endowed with some discrete stability properties.
Numerical approximations 
are naturally interesting for quantitative estimations in applications.
But 
they are also necessary qualitatively at present, if one wants to precisely 
discuss solutions that are physically reasonable.

Note that our new model cannot be straightforwardly treated by the standard FV strategy 
recalled in section~\ref{sec:general} for the sake of clarity. 
Indeed, 
(i) the mathematical entropy (i.e. the energy here, strictly convex with respect to conservative variables) should be replaced with
a Helmholtz free-energy (a priori non convex with respect to conservative variables) to ensure a second-principle in presence of source terms here
(i.e. at non-zero temperature), and
(ii) our model (without source term) is a convex extension of a system with involutions, and we need the exact preservation of those involutions at the discrete level
to ensure stability -- which is a well-known difficulty in multi-dimensional numerical discretizations.

In section~\ref{sec:oned}, we present a modification of the standard strategy.
We suggest to consistently reconstruct at interfaces and in cells 
those (conservative) variables which do not correspond to fundamental conservation laws in physics (unlike mass, energy, momentum).
Additionally, our numerical discretization gives some insight onto the physical meaning of the new system
(with a view to selecting physically reasonable solutions, recall).

We draw a conclusion in Section~\ref{sec:num} after showing some numerical results.


\section{Viscoelastic Saint-Venant 
equations} 
\label{sec:model}

\mycomment{Linear momentum conservation : postulated, or derived (like energy conservation !) from a variational principle ?}

Before introducing our viscoelastic Saint-Venant 
equations for 
Maxwell fluids, 
let us recall the usual Saint-Venant system of equations, for 2D flows of simpler fluids,
i.e. the standard (nonlinear) shallow-water model for thin-layer (i.e. shallow) free-surface flows governed by a hydrostatic pressure.

\subsection{Saint-Venant models for shallow free-surface flows}
\label{sec:saintvenant}

We consider a Eulerian description 
in a Galilean frame equipped with a Cartesian system of coordinates $(\be_x,\be_y,\be_z)$,
under a constant 
gravity field $\bg=-g\be_z$. 
%
The \emph{shallow} free-surface flows of homogeneous incompressible fluids 
above a flat impermeable plane $z=0$
can be modelled with a 2D velocity field $\bU(t,x,y)=U^x(t,x,y)\be_x+U^y(t,x,y)\be_y$ 
for ``infinitesimal fluid columns'' under a 
non-folded free-surface $z=H(t,x,y)\ge0$ 
using mass and momentum balance laws: 
\begin{align}
\label{eq:SVH}
& \partial_t H + \div( H \bU ) = 0 
\,,
\\
\label{eq:SVHUV}
& \partial_t(H\bU) + \div(H\bU\otimes\bU + H(P+\Sigma_{zz})\bI - H\bSigma_h ) = 0 
\,,
\end{align}
where $\bI$ denotes the identity second-order tensor.

Assuming the 
pressure 
hydrostatic $p(t,x,y,z)\approx g(H-z)$
so $P=gH/2\approx\frac1H\int_0^H dz p$ in \eqref{eq:SVHUV},
the 2D velocity 
$\bU=(U,V)$ is 
interpreted as a depth-average 
$$
U^x \approx \frac1H\int_0^H dz\; u^x \qquad U^y \approx \frac1H\int_0^H dz\; u^y
$$
of the horizontal components of a 
3D velocity 
$\bu=u^x\be_x+u^y\be_y+u^z\be_z$,
possibly the mean 
of a statistically stationary field e.g. in turbulent flows.

The ``stress'' term\footnote{
 Note that $P$, $\Sigma_{zz}$ and the components $\Sigma_{h,ij}$ have the dimension of energies per unit mass 
 rather than per unit volume here, unlike standard Cauchy stresses and pressures: 
 one could term them \emph{specific} forces. 
 But for the sake of simplicity, we will 
 omit the label ``specific'', 
 insofar as there is no ambiguity here. 
} $-\Sigma_{zz}\bI+\bSigma_h$ 
is usually symmetric, i.e. $\Sigma_{yx}=\Sigma_{xy}$ in
\beq
\label{eq:horizontalstress}
 \bSigma_h = \Sigma_{xx}\be_x\otimes\be_x + \Sigma_{yy} \be_y\otimes\be_y + \Sigma_{xy} \be_x\otimes\be_y + \Sigma_{yx} \be_y\otimes\be_x
\eeq
and it accounts, depending on the closure, for: 
\begin{itemize}
 \item depth-averaged \emph{Cauchy stresses} when 
\begin{multline}
\label{eq:viscoussteadystate}
\Sigma_{xx}=2\nu\partial_xU^x \quad \Sigma_{yy}=2\nu\partial_yU^y \quad \Sigma_{xy}=\nu(\partial_xU^y+\partial_yU^x)=\Sigma_{yx} 
\\ \Sigma_{zz}=-\nu(\partial_xU^x+\partial_yU^y)
\end{multline}
(the case of Newtonian fluids with constant viscosity $\nu>0$ 
see e.g. \cite{gerbeau-perthame-2001,marche-2007} for a justification based on 
asymptotic analysis of the depth-averaged Navier-Stokes equations for 
Newtonian fluids), or
%
%
%
%
 \item empirical corrections in the depth-averaged acceleration terms
$$ 
H \int_0^H dz\, u^iu^j \approx \left(\int_0^H dz\, u^i\right) \left(\int_0^H dz\, u^j\right)
$$
when the horizontal components $(u^x,u^y)$ of the (time-averaged) velocity 
do not have a uniform profile in vertical direction\footnote{
 The effect is called \emph{dispersion} in hydraulics.
}, 
or 
when the second-order moments 
of a turbulent 
velocity \footnote{
 They are usually termed \emph{turbulent stresses}. 
} cannot be neglected in the dynamics of the (depth-averaged) mean 
velocity field $\bu$. 
\end{itemize}

New 2D models keep being developped for turbulent flows (see e.g. \cite{Teshukov2007,hal-01529497,boyaval-hal-01661269}) 
as well as for real non-Newtonian fluids (see e.g. \cite{bouchut-fernandeznieto-mangeney-narbonareina-2015}).

One approach to the 
modelling of non-Newtonian stresses is to start with 3D models and next close 2D depth-averaged models using scaling assumptions 
(like \cite{gerbeau-perthame-2001,marche-2007} for Newtonian fluids, 
\cite{narbona-reina-bresch-2010} or our previous work \cite{bouchut-boyaval-2015} otherwise). 

But the 3D viscoelastic flows models still raise many questions (see Introduction),
and we believe a direct 2D analysis in the shallow framework could be a useful 
alternative 
first step toward better 3D 
models. 
Let us recall precisely the modelling issue, which still holds for \eqref{eq:SVH},\eqref{eq:SVHUV},
or its more usual formulation with source terms modelling friction over rugous bottom in real 
application~\cite{hervouet-2007}:
\begin{align*}
& \partial_t H + \partial_x( H U^x ) + \partial_y(H U^y) = 0
\\
& \partial_t(HU^x) + \partial_x( HU^xU^x + HP+H\Sigma_{zz}-H\Sigma_{xx} ) + \partial_y( HU^xU^y - H\Sigma_{xy} ) = -KHU^x
\\
& \partial_t(HU^y) + \partial_x( HU^xU^y - H\Sigma_{yx} ) + \partial_y (  HU^yU^y + HP+H\Sigma_{zz}-H\Sigma_{yy} ) = -KHU^y
\end{align*}
whose 
smooth physically-meaningful solutions are expected to satisfy an additional conservation law for a 
free-energy $E = \frac{1}2 \left( |\bU|^2 + e \right)$ with 
$D\ge0$ 
: 
\beq
\label{eq:energyconservation} 
\partial_t(HE) + \div(HE\bU+H(P+\Sigma_{zz})\bU-H\bSigma_h\cdot\bU) = -KH |\bU|^2 - HD
\eeq
or equivalently 
for a specific internal energy $e$:
\beq
\label{eq:internalenergydissipation} 
\partial_t e + (\bU\cdot\grad)e + (P+\Sigma_{zz})\div\bU - \bSigma_h:\grad\bU = -D \,.
\eeq


In the pure viscous case \eqref{eq:viscoussteadystate} (Newtonian fluids),
it is well-known that \eqref{eq:energyconservation} holds as \emph{equality}
with 
$D=2\nu\left(|D(\bU)|^2+2|\div\bu|^2\right)\ge0$ and $e=gH$
for smooth solutions with a Helmholtz free-energy $E$ 
termed the \emph{mechanical energy}
\begin{equation}
\label{mechanicalenergy} 
E = \frac{1}2 \left( |\bU|^2 + gH \right) \,,
\end{equation}
while \eqref{eq:energyconservation} only holds as \emph{inequality} ($\le$) for weak solutions 
that implicitly model irreversible flows \cite{Vasseur2016}. 

In the ideal case $\Sigma_{zz}=0,\bSigma_h=\bzero$ for the widely-used inviscid shallow-water model 
(or Euler 
isentropic 2D flow model for perfect polytropic gases with $\gamma=2$),
the equality \eqref{eq:energyconservation} also holds for smooth solutions to
the symmetric hyperbolic quasilinear system
(note Godunov-Mock theorem applies 
since \eqref{mechanicalenergy} is a strictly convex function of $(H^{-1},\bU)\in\R_{>0}\times\R^2$)
see e.g. \cite{godlewski-raviart-1996,benzonigavage-serre-2007}.
%
%
Moreover, on requiring the 
\emph{inequality} 
associated with \eqref{eq:energyconservation} ($D=0$ in the ideal case):
\beq
\label{eq:energydissipation} 
\partial_t(HE) + \div(HE\bU+H(P+\Sigma_{zz})\bU-H\bSigma_h\cdot\bU) \le -KH |\bU|^2 
\eeq
one can define 
physically-admissible \emph{entropy solutions} 
where irreversible processes dissipate the mechanical energy \eqref{mechanicalenergy},
that are unique in 1D within a translation-invariant solution class 
(see e.g. \cite{bianchini-bressan-2005} when $H>0$ and the shallow-water system is strictly hyperbolic).

None of the two standard cases above model well the propagation of shear stress/strain at finite-speed\footnote{
  The propagation of information at finite-speed is not only 
  physical. It also allows for a precise computation of fronts, that can be compared with experimental observations.
  This is the reason why the inviscid model is more often used in practice, e.g. by hydraulic engineers,
  although it cannot sustain shear motions like the viscous model (i.e. a non-trivial steady state in shear flows).
}. So let us look for a hyperbolic quasilinear model of Saint-Venant type (a shallow-water model) 
for 2D viscoelastic flows. Precisely, we look 
an additional 
law like \eqref{eq:energyconservation} that properly defines a notion of viscosity 
and accounts for vortices in stationary flows.

\subsection{Saint-Venant models generalized to Maxwell fluids}
\label{sec:maxwell}


Viscoelastic fluids of Maxwell type 
are characterized by an elasticity modulus $G$ (in stress units) 
and a finite relaxation time scale\footnote{
 Typically characterizing the time needed for the 
 stress in a fluid initially at rest to relax to a viscous state (proportional to strain-rate by the viscosity factor $\nu=G\lambda>0$) 
 after suddenly straining the fluid at a fixed 
 maintained rate. 
}
$\lambda$ (termed Weissenberg number when non-dimensionalized with a time scale of the flow like $|\grad\bU|^{-1}$).

Closure formulas for the non-Newtonian Cauchy stresses of Maxwell fluids can be obtained 
e.g. following the same depth-averaged 
analysis as in \cite{gerbeau-perthame-2001,marche-2007}, 
starting with full 3D models for free-surface flows of Maxwell fluids
see e.g. \cite{bouchut-boyaval-2015}. 
When the time rate of change for the stress tensor $\bSigma$ is the \emph{upper-convected} 
derivative (i.e. we consider Upper-Convected Maxwell fluids), one obtains
\begin{align}
\label{eq:Sigmah}
& D_t \bSigma_h - \bL_h\bSigma_h - \bSigma_h \bL_h^T = (2\nu\bD_h-\bSigma_h)/\lambda
\\
\label{eq:Sigmazz}
& D_t \Sigma_{zz} + 2\div\bU \Sigma_{zz} = (-2\nu\div\bU-\Sigma_{zz})/ \lambda
\end{align} 
without background viscosity (i.e. with zero retardation time in viscoelastic terminology)
in some asymptotic regime,
denoting $\bL_h$ the horizontal velocity gradient in \eqref{eq:Sigmah}
which corresponds to the horizontal stress tensor \eqref{eq:horizontalstress}.

Viscous stresses with a viscosity $\nu>0$ arise 
from \eqref{eq:Sigmah},\eqref{eq:Sigmazz} in 
\eqref{eq:SVHUV} when $G=\lambda\nu$, $\lambda\to0$.
On the contrary, when $\lambda\to\infty$, the fluid becomes purely elastic and governed by
the homogeneous quasilinear system \eqref{eq:SVUCM} which is similar 
to a thin-layer approximation 
of the 
elastodynamics system governing the 
deformations of a Hookean hyperelastic continuum (see details below).


With $\bB_h:=\lambda\bSigma_h/\nu+\bI$ and $B_{zz}:=\lambda\Sigma_{zz}/\nu+1$,
\eqref{eq:SVH},\eqref{eq:SVHUV},\eqref{eq:Sigmah},\eqref{eq:Sigmazz} rewrites
as the \emph{hyperbolic} quasilinear system \eqref{eq:SVUCM} (see App.~\ref{app:hyperbolicity} for details and a proof):
\beq
\boxed{
\label{eq:SVUCM}
\begin{aligned}
& \partial_t H + \div( H \bU ) = 0
\\
& \partial_t(H\bU) + \div(H\bU\otimes\bU + (gH^2/2 + GH B_{zz})\bI - GH\bB_h ) = -KH\bU
\\
& \partial_t \bB_h + \bU\cdot\grad \bB_h 
- \bL_h\bB_h - \bB_h \bL_h^T 
= (\bI-\bB_h)/\lambda
\\ 
& \partial_t B_{zz} + \bU\cdot\grad B_{zz} 
+ 2 
B_{zz} \div\bU = (1-B_{zz}) / \lambda
\end{aligned}
}
\eeq
which is 
similar to another hyperbolic system for viscoelastic flows already known in the literature, see Rem.~\ref{rem:similarity}.

But hyperbolicity is 
only a necessary condition for the definition of 
a sensible initial-value problem 
with a 
quasilinear system 
(in a very weak sense, see e.g.~\cite{benzonigavage-serre-2007} in the particular case of constant coefficients),
it is 
not sufficient. 
The initial-value problem 
can be shown well-posed (on small times $t\in[0,T)$ for smooth initial data)
for the \emph{symmetric hyperbolic} quasilinear systems 
\cite{Kato1975,majda1984book}.
For instance, the 
systems of \emph{conservation laws} 
are symmetric hyperbolic 
when they are endowed with an additional conservation law for a strictly-convex functional termed ``entropy'' \cite{godunov-1959,godlewski-raviart-1996}.

Although the quasilinear SVUCM system is endowed with an additional 
conservation law for a convex free-energy functional (see $E$ below in \eqref{eq:energy}),
\eqref{eq:SVH},\eqref{eq:SVHUV},\eqref{eq:SVUCM} is obviously not in conservation form.
So, the additional conservation law is not useful,
and the meaning of 
weak solutions remains unclear. 

One could try to make sense of SVUCM as such with non-conservative products like in \cite{LeFloch1989,berthon-coquel-lefloch-2011}. 
However, note that SVUCM is unlikely to possess ``generalized symmetrizers'', 
see the analysis 
in \cite{olsson-ystrom-1993} for a close system in 2D, %
while ``generalized symmetrizability'' seems a minimum requirement 
(see e.g. \cite{MR3322782}) to define a meaningful concept of solution
(like dissipative measure-valued solutions satisfying a weak-strong uniqueness principle as in \cite{demoulini-stuart-tzavaras-2012}).
%

Besides, let us also recall that the question how to correctly formulate (multidimensional) equations for flows of non-Newtonian fluids is not settled in general,
especially for \emph{compressible} viscoelastic flows \cite{Bollada2012}, although it has received a number of answers
(the close 2D system discussed in \cite{olsson-ystrom-1993} 
 see also Rem.~\ref{rem:similarity} 
 is for \emph{slightly compressible} viscoelastic fluids). 

That is why we propose here to consider a 
modification of \eqref{eq:SVH},\eqref{eq:SVHUV},\eqref{eq:SVUCM} that bears the same physical meaning
but conforms with the existing theory for symmetric hyperbolic system of conservation laws.
Our strategy is to devise an enlarged system of conservation laws 
that contains \eqref{eq:SVH},\eqref{eq:SVHUV},\eqref{eq:SVUCM}.


To that aim, observe first that \emph{when $\lambda\to\infty$}, 
$\bB_h$ and $B_{zz}$ in fact identify with the horizontal and vertical components (resp.) 
of the Cauchy-Green (left) deformation tensor $\bB=\bF\bF^T$ 
in a hyperelastic incompressible homogeneous continuum,
where $\bF=\partial_{a,b,c}(x,y,z)$ 
is the deformation gradient 
with respect to a reference configuration in the Cartesian coordinate system $(\be_a,\be_b,\be_c)$.

Precisely, 
the Cauchy stress terms in the momentum equation 
rewrite 
$$
H\bSigma_h = |\bF_h|^{-1}\partial_{\bF_h}\left(\frac{G}2\bF_h:\bF_h\right)\bF_h^T
\text{ and } 
H\bSigma_{zz} = |F^z_c|\partial_{F^z_c}\left(\frac{G}2|F^z_c|^2\right)F^z_c
$$ 
when $H\equiv F^z_c=|\bF_h|^{-1}>0$.
So, when $\lambda\to\infty$, eq. \eqref{eq:SVUCM} in the SVUCM system is a consequence
of the elastodynamics of \emph{Hookean} incompressible materials with uniform mass density 
\emph{in the shallow-water regime}, 
using
\begin{equation}
\label{eq:defgradient} 
\partial_t \bF + (\bu\cdot\grad) \bF = \bL \bF \,.
\end{equation}
%
In fact, the SVCUM system \eqref{eq:SVH},\eqref{eq:SVHUV},\eqref{eq:SVUCM} 
\emph{coincides} 
with the elastodynamics of a 
2D hyperelastic continuum, 
with $\bB_h\equiv\bF_h\bF_h^T$ computed from $\bF_h$
\begin{equation}
\label{eq:defgradienth} 
\partial_t \bF_h + (\bu\cdot\grad) \bF_h = \bL_h \bF_h
\end{equation}
when $\lambda\to\infty$,  recall \eqref{eq:defgradient} as well as $B_{zz}\equiv|\bF_h|^{-2}$ by incompressibility.
The energy conservation law \eqref{eq:energyconservation} is 
satisfied by the Helmholtz 
\emph{polyconvex} free-energy
\begin{equation}
\label{eq:energy} 
E = (U^2+V^2)/2+E_H+E_\Sigma 
\end{equation}
with $D=0$ (recall $H=|\bF_h|^{-1}$; $\lambda\to\infty$ means no source) and 
internal energy
\begin{equation}
\label{eq:energypolyconvex} 
e \equiv E_H+E_\Sigma := \frac{g}2|\bF_h|^{-1} + \frac{G}2 ( \bF_h:\bF_h + |\bF_h|^{-2} ) \,.
\end{equation}
This is consistent with the fact that Cauchy stress term $H\bSigma_h-H\bSigma_{zz}\bI\equiv GH(\bB_h-B_{zz}\bI)$ in the momentum equation 
actually equals $H(\partial_{\bF_h}e)\bF_h^T$.


So, at least when $\lambda\to\infty$, one can make sense of SVUCM system 
to 
model time-evolutions.
SVUCM coincides with the 2D elastodynamics of a hyperelastic incompressible materials,
which is equipped with a polyconvex energy like \eqref{eq:energypolyconvex}
and 
a symmetric hyperbolic conservative formulation
(see e.g. \cite{dafermos-2000,wagner-2009}):
%
the Cauchy problems 
are well-posed on small times given smooth initial conditions.

\smallskip

For the general case with 
$\lambda>0$, we now propose 
to embed SVUCM 
into a quasilinear system of \emph{conservation laws}
for the \emph{2D visco-elastodynamics} of a hyperelastic incompressible continuum with memory.
To that aim, we introduce new state variables\footnote{
 One may also want to call them ``internal'' variables. 
} $\bA_h=\bA_h^T>0$, $A_{cc}>0$ 
that 
account for 
``viscous'' deformations of the 
microstructure 
in a coordinate system attached to the reference configuration,
which is reminiscent of the distortion metric used in elasto-plasticity \cite{NARDINOCCHI201334,lee-1969}.
We postulate 
simple constitutive 
laws (see Remark~\ref{rem:kinetic}):
\begin{align}
\label{eq:Ah}
& D_t \bA_h = (\bF_h^{-1}\bF_h^{-T}-\bA_h)/\lambda \,,
\\
\label{eq:Acc}
& D_t A_{cc} = (H^{-2}-A_{cc})/\lambda \,. 
\end{align}
%
%
Then, 
using in \eqref{eq:energy} 
the internal energy 
of SVUCM (recall~\cite{wapperom-hulsen-1998-a,bouchut-boyaval-2013} e.g.)
\begin{equation}
\label{eq:internalenergy} 
E_H+E_\Sigma = \frac{g}2 H + \frac{G}2 \left( \tr(\bB_h)-\ln\left(\det\bB_h\right) + B_{zz}-\ln(B_{zz}) \right) \,,
\end{equation}
with horizontal (symmetric) and vertical positive strains 
defined 
by
\beq
\label{eq:BhandBzz}
\bB_h=\bF_h\bA_h\bF_h^T
\quad
B_{zz}=H^2 A_{cc}>0
\,,
\eeq
the mass, momentum and energy conservation laws lead to the 
system: 
\beq
\hspace{-.5cm}
\boxed{
\label{eq:SVUCM0}
\begin{aligned}
& \partial_t H + \div( H \bU ) = 0
\\
& \partial_t (H\bF_h) + \div( H \bU\otimes\bF_h-H\bF_h\otimes\bU ) = 0
\\
& \partial_t(H\bU) + \div(H\bU\otimes\bU + (g\tfrac{H^2}2  + GH^3 A_{cc})\bI - GH\bF_h\bA_h\bF_h^T ) = -KH\bU
\\
& \partial_t (H\bA_h) + \div( H \bU\otimes\bA_h ) = H(\bF_h^{-1}\bF_h^{-T}-\bA_h)/\lambda
\\ 
& \partial_t (HA_{cc}) + \div( H \bU A_{cc} ) =  H(H^{-2}-A_{cc})/\lambda 
\end{aligned}
}
\eeq
which we term Saint-Venant-Maxwell or SVM in short.
\begin{proposition}
\label{prop:symhyp}
The quasilinear system of conservation laws \eqref{eq:SVUCM0} 
written for $(H,H\bU,H\bF_h,HA_{cc}^{1/4},H\bA_h^{-2})$
with ($\bA_h^{-1}$ is $\bA_h^{-2}$ square-root matrix):
\beq
\begin{aligned}
\label{eq:SVUCM01}
& \partial_t (H\bA_h^{-2}) + \div( H \bU\otimes\bA_h^{-2} ) 
= H\bA_h^{-1}(\bI-\bA_h^{-1}\bF_h^{-1}\bF_h^{-T}+\bF_h^{-T}\bF_h^{-1}\bA_h^{-1})\bA_h^{-1})/\lambda \,, 
\\ 
& \partial_t (HA_{cc}^{1/4}) + \div( H \bU A_{cc}^{1/4} ) =  H(H^{-2}A_{cc}^{-3/4}-A_{cc}^{1/4})/4\lambda \,,
\end{aligned}
\eeq
is equipped (on neglecting the source) with a mathematical entropy $H\tilde E$ where
\beq
\label{eq:tildeE}
\tilde E = (|\bU|^2+gH)/2 + {G} \left( \tr(\bF_h\bA_h\bF_h^T) + H^2 A_{cc} \right)/2 \,.
\eeq
It is therefore \emph{symmetric hyperbolic}
on the convex 
admissibility domain
$$
\Acal := \{H>0\,,\ \bA_h^{-1}=\bA_h^{-T}>0\,, A_{cc}^{-1}>0\} \,.
$$
\end{proposition}
\begin{proof}
Since $H\tilde E$ obviously satisfies \eqref{eq:energyconservation} in the case without source,
it suffices to show that \eqref{eq:SVUCM0} is symmetric hyperbolic on 
$\Acal$ (obviously convex), 
i.e. that $H\tilde E$ 
is \emph{stricly} convex with respect to a full 
set of conserved variables, 
using Godunov-Mock theorem \cite{godlewski-raviart-1996}.
Moreover, $H\tilde E$ is (strictly) convex in $(H,H\bU,H\bF_h,HA_{cc}^{1/4},H\bA_h^{-2})$
if and only if $\tilde E$ is (strictly) convex in $(H^{-1},\bU,\bF_h,A_{cc}^{1/4},\bA_h^{-2})$ \cite{bouchut-2003},
i.e. if $\tilde E_1 = gH + G \left( H^2 A_{cc} \right)$ and 
$\tilde E_2 = \tr(\bF_h\bA_h\bF_h^T)$ are 
(strictly) convex in $(H^{-1},A_{cc}^{1/4})$ and $(\bF_h,\bA_h^{-2})$
respectively, like $|\bU|^2/2$ in $\bU$.
Now, 
the smooth function $\tilde E_1$ is strictly convex insofar as its Hessian matrix is 
strictly positive:
$$
\nabla^2_{H^{-1},A_{cc}^{-1}} \tilde E_1 = 
\begin{pmatrix}
2gH^3 + 6GH^4 A_{cc} & - 2G H^3 A_{cc}^{3/4}                
\\
- 2G H^3 A_{cc}^{3/4} & 2G H^2 A_{cc}^{1/2}
\end{pmatrix}
\,.
$$
%
On the other hand,
consider two couples of matrix values $(\bF_1,\bY_1:=\bA_1^{-2})$ and $(\bF_2,\bY_2:=\bA_2^{-2})$ for $(\bF_h,\bA_h^{-2})$.
for any $\theta\in[0,1]$, using $\bF_\theta=\theta\bF_1+(1-\theta)\bF_2$, $\bY_\theta=\theta\bY_1+(1-\theta)\bY_2$ 
and $\bH_\theta = \theta\bF_1\bY_1^{-\frac14}+(1-\theta)\bF_2\bY_2^{-\frac14}$,
$\bD_\theta \bF_1\bY_1^{-\frac14}\bY_2^{\frac14}-\bF_2\bY_2^{-\frac14}\bY_1^{\frac14}$
it holds
\begin{multline}
\tr(\bH_\theta \bY_\theta^{\frac12} (\bH_\theta)^T)
> \tr\left( \bH_\theta (\theta\bY_1^{\frac12}+(1-\theta)\bY_2^{\frac12}) (\bH_\theta)^T \right)
\\
= \tr\left( (\theta^2+\theta(1-\theta)\bY_1^{-\frac14}\bY_2^{\frac12}\bY_1^{-\frac14})\bF_1^T\bF_1 
+ ((1-\theta)^2+\theta(1-\theta)\bY_2^{-\frac14}\bY_1^{\frac12}\bY_2^{-\frac14})\bF_2^T\bF_2 \right)
\\
= \tr\left( \bF_\theta^T\bF_\theta + \theta(1-\theta) \bD_\theta^T \bD_\theta \right)
\ge \tr( \bF_\theta^T\bF_\theta )
\end{multline}
hence $\tr(\bH_\theta (\bH_\theta)^T) > \tr( \bF_\theta \bY_\theta^{-\frac12} \bF_\theta^T )$
since $\bY_\theta^{-\frac12}$ is symmetric positive definite, 
which is the desired result $\theta E_2(\bF_1,\bY_1) + (1-\theta) E_2(\bF_2,\bY_2)>E_2(\bF_\theta,\bY_\theta)$.
\end{proof}
Proposition~\ref{prop:symhyp} allows one to check that the system \eqref{eq:SVUCM0} makes sense as a 
flow model with any smooth source term, for small times 
at least.
But in fact, the full SVM system \eqref{eq:SVUCM0} can also be shown \emph{thermodynamically compatible} i.e.:
\begin{corollary}
\label{cor:strong}
Given smooth initial conditions, the Cauchy problems for strong solutions to \eqref{eq:SVUCM0} are well-posed.
These strong solutions 
preserve the relation $H\equiv F^z_c=|\bF_h|^{-1}$.
They also satisfy the companion conservation law \eqref{eq:energyconservation} for the same free-energy 
\eqref{eq:energy} as SVUCM where
the usual hydrostatic potential term $E_H=gH/2$ of Saint-Venant 
is complemented by the viscoelastic term 
\beq
\label{eq:viscoelasticpotential}
E_\Sigma \equiv \frac{G}2 \left( \tr(\bF_h\bA_h\bF_h^T) + H^2 A_{cc} - \ln(|\bF_h|^2\det\bA_h H^2 A_{cc}) \right)
\eeq
and by \emph{thermodynamically compatible} source terms, dissipating energy at rate:
\beq
\label{eq:dissipation}
D 
\equiv G(\tr\bB_h+\tr\bB_h^{-1}-2\tr\bI+B_{zz}+B_{zz}^{-1}-1)/(2\lambda) > 0 \,.
\eeq
\end{corollary}
\begin{proof}
Corollary~\ref{cor:strong} is a 
consequence of Proposition~\ref{prop:symhyp}:
the well-posedness of Cauchy problems for strong solutions to \emph{symmetric hyperbolic} systems is classical, see e.g. \cite{benzonigavage-serre-2007}.
Then, given smooth solutions, one can directly check that $|\bF_h|^{-1}$ follows the same evolution as $H$.
The SVUCM system can be 
retrieved exactly using \eqref{eq:BhandBzz}, \eqref{eq:Ah}, \eqref{eq:Acc}.
This is the reason why the same companion conservation law \eqref{eq:energyconservation} 
holds for \eqref{eq:SVUCM0} as for SVUCM,
where the source term $D$ is thermodynamically compatible -- and also an upper bound for $E_\Sigma$, thus $E$ ! 
\end{proof}

Some comments about the structure of SVM are now in order.

\smallskip

First, note that we have not been able to find
a full set of 
conserved variables 
such that the free-energy $HE$ is convex \emph{strictly} on the whole admissible domain $\Acal$
when 
it is defined as in \eqref{eq:energy} 
by 
$e \equiv E_H+E_\Sigma$, $E_H=\frac{g}2H$ 
and \eqref{eq:viscoelasticpotential} for $E_\Sigma$.
This is why we use $H\tilde E$ rather than $HE$ to show that the SVM system is symmetric hyperbolic and thus a good model, with well-posed Cauchy problems.
%
%
%
But we insist on the importance of $HE$ with $E_\Sigma$ defined by \eqref{eq:viscoelasticpotential}.
It allows one to show that the source terms in SVM which formally yield 
Navier-Stokes equilibrium asymptotically when $\lambda\to0$ are thermodynamically compatible. 
%
%
A result like Corrolary~\ref{cor:strong} is necessary, 
to justify 
weak solutions 
satisfying the \emph{inequality}:
\begin{multline}
\label{eq:SVHE}
\partial_t(HE) + \partial_x\left( HEU + H(P+\Sigma_{zz}-\Sigma_{xx})U - H\Sigma_{xy}V \right) 
\\
+ \partial_y\left( HEV - H\Sigma_{yx}U + H(P+\Sigma_{zz}-\Sigma_{yy})V \right) \le -KH |\bU|^2 - HD
\end{multline}
as 
second thermodynamics principle complementing \eqref{eq:SVUCM0} for time-evolution with irreversible processes
(the non-convexity of the free-energy is physical in presence of a source which modifies the energy landscape at a fixed temperature).
%
%
%
%

\smallskip

Second, recall that the SVUCM system can be 
retrieved exactly from (smooth solutions) to SVM using \eqref{eq:BhandBzz}, \eqref{eq:Ah}, \eqref{eq:Acc} (cf. 
Corr.~\eqref{cor:strong}).
Now, the SVUCM system with 
$\bB_h\equiv\bF_h\bA_h\bF_h^T$, $B_{zz}=H^2 A_{cc}$ is a closed Galilean-invariant 
quasilinear system which is \emph{hyperbolic} (strongly, see App.~\ref{app:hyperbolicity}).
The Jacobian matrix (say for 1D waves along $\be_x$) is diagonalizable. The 7 eigenvalues
$$
\lambda_0 = \ux \ (\times 3)\,, 
\
\lambda_{1\pm} = \ux \pm \sqrt{G\sx}
\,,\
\lambda_{2\pm} = \ux \pm \sqrt{g\ro+G(3\sx+\sz)} \,,
$$
with as many eigenvectors are also 7 eigenvalues of the extension termed SVM. 
Moreover, the eigenvalue $\lambda_0 = \ux$ 
is of multiplicity $5$ at least for SVM, when using $\bA_h$ as independent variable 
($B_{zz}$ is equivalent to $A_{cc}$ as long as $H=|\bF_h|^{-1}$).
In fact, being symmetric hyperbolic, our SVM system above is also strongly hyperbolic when ``freezing'' the Jacobian:
it is easily seen that $0$ is the remaining (real) eigenvalue of SVM 
(above, in Eulerian coordinates), with multiplicity 2.
So, in general, 
\emph{resonance} can occur \cite{isaacson-temple-1992,bouchut-2004} 
as well as numerical difficulties like for 
systems 
with discontinuous 
flux, see e.g. \cite{MR1278994}. 
This difficulty was not apparent in our previous work on the (closed) 1D subsystem \cite{bouchut-boyaval-2013}
but it is inline with standard results in multidimensional elastodynamics (the case $\bA_h=\bI$, $A_{cc}=1$),
see e.g. \cite{wagner-2009,kluth-despres-2010} and references therein,
and possibly with the numerical difficulties observed close to vacuum $H=0$ in our previous 1D works on SV(UC)M \cite{boyaval-2018}.


\smallskip

Last, for relevant applications to a specific context, 
one may still want to precise 
the physical interpretation of the new state variable $\bA$.
At present, we think of it as a material parameter that simply follows the flow at ``zero-temperature'',
and that relaxes to balance mechanical deformations otherwise.
For application to polymer suspensions, one may think of it as a mean-field approximation, see Rem.~\ref{rem:kinetic} for 
a 
``micro-macro'' interpretation.
More generally, $\bA$ is some kind of viscous strain measuring the \emph{distortion of the 
volume elements} (or equiv. \emph{microstructure deformations}) in a coordinate system attached to a reference configuration.
This is mathematically different from, but similar in spirit to, the \emph{tensor} state variable 
$\bA$ introduced in \cite{Peshkov2016}
(governed by the non-conservative equations \eqref{eq:Sigmah},\eqref{eq:Sigmazz} like the elastic strain measure $\bB=\bF\bA\bF^T$ there).
%
%
In any case, the fact that $\bA$ describes material 
properties of the flows 
can be entlightened on writing SVM with a Lagrangian description,
when a bijective flow map $\phi:(t,x,y)\to(t,a,b)$ between the current and reference configurations
is well-defined with 
$\bU\equiv\partial_t\phi^{-1}$ and 
$\bF\equiv\grad\phi^{-1}$.
Recall 
\beq\label{eq:SVUCM0detail}
\hspace{-.3mm}\begin{aligned}
& \partial_t H + \partial_{j}( H U^j ) = 0
\\
& \partial_t (HF^i_\alpha) + \partial_j( H U^j F^i_\alpha - H F^j_\alpha U^i ) = 0
\\
& \partial_t (H U^i) + \partial_j( H U^jU^i + (g\tfrac{H^2}2 + GH^3 A_{cc})\delta_{i=j} - GH F^i_\alpha A_{\alpha\beta} F^j_\beta) = -KH U^i
\\
& \partial_t (H A_{\alpha\beta}) + \partial_j( H U^j A_{\alpha\beta} ) 
= H ( |\bF_h|^{-2} \sigma_{\alpha\alpha'}\sigma_{\beta\beta'} 
F^k_{\alpha'} F^k_{\beta'} -A_{\alpha\beta})/\lambda
\\ 
& \partial_t (HA_{cc}) + \partial_{j}( H U^j A_{cc} ) =  H(H^{-2}-A_{cc})/\lambda
\end{aligned}
\eeq
is the 
Eulerian description \eqref{eq:SVUCM0} 
for 
$(H,F^i_\alpha,U^i,A_{\alpha\beta},A_{cc})$ ($i,j\in\{x,y\} ; \alpha,\beta\in\{a,b\}$) in a Cartesian basis
where 
$\bF_h = F^i_\alpha \be_i\otimes\be_\alpha$ and $\bA_h = A_{\alpha\beta} \be_\alpha\otimes\be_\beta$ 
so, denoting $\sigma_{xy}=1=-\sigma_{yx}$, it holds $\bF_h^{-1}=|\bF_h|^{-1} (\sigma_{ij}\sigma_{\alpha\beta} F^j_\beta) \be_\alpha\otimes\be_i$
with $|\bF_h|= \sigma_{ij}\sigma_{\alpha\beta} 
F^{i}_{\alpha} F^{j}_{\beta}$.
Then, an equivalent\footnote{In the case of sufficiently smooth solutions 
\cite{wagner-1994,wagner-2009}.} 
Lagrangian description holds using 
\begin{equation}
 \label{eq:piola}
\partial_{\alpha}( \sigma_{\alpha\beta} F^k_\beta ) = 0
\quad
\forall \gamma
\quad 
\text{ or }
\quad
\partial_{j}( H F^j_\gamma ) = 0 
\quad 
\forall k 
\,,
\end{equation}
if $H|\bF_h|=1$ i.e. the so-called \emph{Piola's identities}, 
which reads 
: 
\beq
\label{eq:SVUCM0detaillag}
\begin{aligned}
& \partial_t H^{-1} - \partial_{\alpha}( U^j \sigma_{jk}\sigma_{\alpha\beta} 
F^k_\beta ) = 0 
\\
& \partial_t F^i_\alpha - \partial_\alpha U^i = 0
\\
& \partial_t U^i + \partial_\alpha \left( (gH^2/2 + GH^3 A_{cc}) \sigma_{ij}\sigma_{\alpha\beta} 
F^j_\beta - G F^i_\beta A_{\beta\alpha} \right) = -K U^i
\\
& \partial_t A_{\alpha\beta} = ( |\bF_h|^{-2}\sigma_{\alpha\alpha'}\sigma_{\beta\beta'} 
F^k_{\alpha'} F^k_{\beta'} -A_{\alpha\beta})/\lambda
\\ 
& \partial_t A_{cc} = (H^{-2}-A_{cc})/\lambda
\end{aligned}
\eeq
while 
\eqref{eq:SVHE} 
for $E =\frac12\left( \sum_i|U^i|^2 + gH + G F^i_\alpha A_{\alpha\beta} F^i_\beta + G H^2 A_{cc} - \log(A_{cc}|\bA_h|) \right)$
\begin{multline}
\label{eq:SVHEeul}
\partial_t( H E ) + \partial_j\left( 
 U^i \left( HE + \left( \frac{g}2H + GH^2 A_{cc} \right) \delta_{i=j} 
 - G F^i_\alpha A_{\alpha\beta} F^j_\beta \right)
\right) 
\\ 
\le -KH|\bU|^2 -HD 
\end{multline}
(simplified with $H|\bF_h|=1$) also has a conservative Lagrangian equivalent:
\beq
\label{eq:SVHElag}
\partial_t E + \partial_\alpha\left(  U^i \left( (\frac{g}2H^2 + GH^3 A_{cc}) \sigma_{ij}\sigma_{\alpha\beta} F^j_\beta - G F^i_\alpha A_{\alpha\beta} \right) \right)
\le -K|\bU|^2 -D
\eeq
with the same dissipation $D>0$ given in \eqref{eq:dissipation}.
Introducing $\Gcal^i_\alpha= G F^i_\beta A_{\alpha\beta}$, 
$$
\Vcal_\alpha = U^i \sigma_{\alpha\beta}\sigma_{ij} F^j_\beta 
\,,\quad
\Pcal^i_\alpha = \Pcal \sigma_{\alpha\beta}\sigma_{ij} F^j_\beta - \Gcal^i_\alpha
\,,\quad
\Pcal= \frac{gH^2}2 + GH^3 A_{cc}
\,,
$$
we obtain a simple reformulation of (the 3 first lines of) \eqref{eq:SVUCM0detaillag} 
and 
\eqref{eq:SVHElag}
as: 
\beq
\label{eq:SVUCM0detaillagbis}
\begin{aligned}
& \partial_t H^{-1} - \partial_{\alpha} \Vcal_\alpha = 0
\\
& \partial_t F^i_\alpha - \partial_\alpha U^i = 0
\\
& \partial_t U^i + \partial_\alpha \Pcal^i_\alpha = -K U^i
\end{aligned}
\eeq
\beq
\label{eq:SVHElagbis}
\partial_t E + \partial_\alpha\left(  U^i \Pcal^i_\alpha \right) \le -K|\bU|^2 -D \,.
\eeq
which can now be easily compared to the usual Lagrangian formulation of elastodynamics \cite{dafermos-2000,wagner-2009} (see Rem.~\ref{rem:lagrangian1}):
$G\bA_h$, $GA_{cc}$ can be understood as variable 
anisotropic elastic properties, which induce a viscous behaviour through friction on a time-scale $\lambda\to0$
inline with Maxwell ideas \cite{Maxwell01011867,maxwell-1874,poisson-1831}.
Furthermore, like in standard elastodynamics, 
the Lagrangian equations above should be useful for variational calculus with SVM (see e.g.~\cite[Lecture 2]{marsden-1981}) 
as well as for numerical approximation (see e.g. \cite{MR2337740,DespresMazeran2005} 
and our last section). 
Note however that, also like in standard elastodynamics, the 
link between the Eulerian and Lagrangian formulations requires $H|\bF_h|=1$ as well as \eqref{eq:piola}.
While this is known to hold for all times 
if it holds initially at a continuous level (this is called an involution),
it is a well-known difficulty 
at a discrete level \cite{kluth-despres-2010}. 


\mycomment{
Finally, although the new state variable $\bA$ can have a precise physical interpretation only in context
(flows with dilute particles 
for instance, recall Rem.~\ref{rem:kinetic}), 
it will appear below that it can have a generic 
mathematical influence on solutions whatever the context, for the definition 
of 
discretizations in particular on noting \eqref{eq:rotationinvariantflux} (see Sec.~\ref{sec:oned}).
}

\mycomment{We use $\bF$ as variable for discretization, and not $\bG$ (a bijective nonlinear transformation of $\bF$), 
because energy conservation is a priori more naturally expressed in Lagrangian setting for $\bU$ and $\bF$,
but $\bG$ might still be useful -- in the Eulerian setting, to build an ALE mapping ?? --}

\begin{remark} 
\label{rem:kinetic}
The system \eqref{eq:SVUCM0} cannot be retrieved from
the standard kinetic interpretation in statistical physics,
when non-Newtonian stresses $\bSigma_h=G(\bB_h-\bI)$ and $\Sigma_{zz}=G(B_{zz}-1)$ are 
due to 
Brownian 
elastic ``dumbbells'' diluted in a fluid suspension 
with conformation matrix 
$\bB=\EE(\bR\otimes\bR)$, $\bR(t,\bx)$ being a random end-to-end \emph{vector} solution to an overdamped Langevin equation 
\cite{ottinger-1996-Book,doi-edwards-1998}.
However, in a similar spirit, one could interpret $\bA_h$ and $A_{cc}$ as the mean-field approximations
$\mathbb{E}(\bG_h\bG_h^T)$ and $\mathbb{E}(|G^c_{z}|^2)$ of 
stochastic 
processes 
:
\begin{equation}
\label{eq:langevin1} 
d \bG_h = \left( - (\bU\cdot\grad) \bG_h - \frac2{\lambda} \bG_h \right) dt + \frac1{2\sqrt\lambda} \bF_h^{-1} d\bW_h(t) 
\end{equation}
\begin{equation}
\label{eq:langevin2} 
d G^c_{z} = \left( - (\bU\cdot\grad) |G^c_{z}| - \frac2{\lambda} |G^c_{z}| \right) dt + \frac1{2\sqrt\lambda} H^{-1} d W^c_z(t) 
\end{equation}
modelling (relaxation of) the 
elastic microstructure 
with thermal fluctuations. 
\end{remark}

\mycomment{
Can one justify the dissipative mean-field theory from a 
variational principle, stochastic then a priori ?
Why should noise act on the deformation gradient then ? (recall it acts on velocity, not position in Langevin theory)

In any case, the interpretation is the door open to many variations: 
different Jaumann derivative if we consider a formulation of the stress in $\bB^{-1}$ ?
FENE-P ?
}


\begin{remark}[Direct derivation of the 
model in  
Lagrangian description]
\label{rem:lagrangian1}
The Lagrangian formulation \eqref{eq:SVUCM0detaillag} of our model could also be straightforwardly derived as
the 2D visco-elastodynamics of a hyperelastic incompressible continuum with internal energy \eqref{eq:internalenergy} 
in a coordinate system attached to a reference configuration (a Lagrangian description), 
see e.g. \cite{wagner-2009,dafermos-2000,zbMATH01293055-qin-1998}.
Like the Eulerian formulation, it can be seen as a formal thin-layer approximation of the 
3D (visco)elastodynamics of an incompressible continuum 
with memory 
and a free-surface 
when the terms containing $F^z_\alpha$ are negligible: the shallow-water regime.
\end{remark}

\mycomment{JUSTIF BY VAR PRINCIPLE~?}
 
\begin{remark}[Hyperbolic models of viscoelastic flows] 
\label{rem:similarity}
In introduction, we have already mentionned that we were aware of a few other interesting hyperbolic models of viscoelastic flows for Maxwell fluids.
The 2D model for slightly compressible flows in \cite{PHELAN1989197,EDWARDS1990411}, see also \cite{olsson-ystrom-1993,Guaily2010158,Guaily2011258} for simulations,
is similar to \eqref{eq:SVUCM} (though it does not conserve mass).
It is hyperbolic under the same physically-natural conditions $H,B_{zz},\tr\bB_h,\det\bB_h\ge0$.
Another 
hyperbolic model for 3D flows of Maxwell fluids is in \cite{Peshkov2016}.
Although similar in spirit, noet that it does not seem to explicitly compare with our new proposition.
Last, note that one can derive another hyperbolic closed subsystem of SVM, similar but different from the SVUCM equations \eqref{eq:SVUCM},
on choosing $-\bF_h$ as a conservative variable in \eqref{eq:SVUCM0}. The energy functional remains the same however.
Without source terms, that model was discovered independently by \cite{Teshukov2007}.
It can be interpreted as a building block for Reynolds-avergaged turbulence models in Saint-Venant framework
(see SVTM in \cite{boyaval-hal-01661269}, and references therein).
\end{remark}


\section{Discretization by a Finite-Volume 
method}
\label{sec:FV}

To investigate quantitatively the 
features of 
\eqref{eq:SVUCM0} 
as a model 
for viscoelastic flows 
(under gravity in the shallow-water regime), we need to define precisely how to compute solutions.
To that aim, we consider \emph{discrete} solutions to Cauchy 
problems 
using a \emph{Finite-Volume} (FV) method 
\cite{godlewski-raviart-1996,leveque-2002}. 

The standard FV strategy requires one to handle 1D Riemann problems (to define numerical fluxes), but this is not precise enough.
Indeed, the discrete FV fields $H,\bF_h,\bU,\bA_h,A_{cc}$ 
(piecewise-constants on the polygonal cells $V_i$, $i\in \NN$ of 2D meshes)
should be ``stable enough'' and satisfy 
not only the conservation laws \eqref{eq:SVUCM0}, 
but also the preservation of the 
domain $H,\bA_h=\bA_h^T,A_{cc}>0$ and the physically-meaningful 
\emph{inequality} \eqref{eq:SVHE}. 
%

In 
section\ref{sec:general} we recall how to standardly build 
stable Riemann-based FV approximations
for a quasilinear model
\begin{equation}
\label{eq:quasilinear}
\partial_t q + \grad_qF_i(q)\partial_i q = B(q)
\end{equation}
like \eqref{eq:SVUCM0} when a companion 
law 
with a dissipation\footnote{
 Typically induced by a thermodynamically-compatible source term 
 like $B_i(q)=(q_i^\infty(q)-q_i)/\lambda_i$ 
 where $q^\infty(q)$ lies in the convex domain for $q$, and $\lambda_i(q)>0$ \cite{chen-levermore-liu-1994}.
} $D(q)\equiv -\nabla_q S(q)\cdot B(q) \ge 0$
\begin{equation}
\label{eq:secondprinc}
\partial_t S(q) + \div 
\bG(q) = -D(q)
\end{equation}
holds for a true \emph{mathematical entropy} $S(q)$ 
convex in the Galilean-invariants of the state $q$ on the whole 
domain $H,\bA_h=\bA_h^T,A_{cc}>0$ \cite{godlewski-raviart-1996,cances-mathis-seguin-2016}.
Typically:
\begin{itemize}
 \item the Cauchy problems are numerically solved by 
 \emph{time-splitting}:
 at each time step, the nonlinear flux terms 
 are computed first by a \emph{forward} method,
 while source terms are computed last \emph{backward} in time, 
 \item in the first (forward) fractional step 
 one uses \emph{1D Riemann solutions} 
 e.g. with 
 \emph{relaxed} conservation laws 
 such that the inequality \eqref{eq:secondprinc} 
 is approximated at the same time as \eqref{eq:quasilinear}. 
\end{itemize}
Such ``standard'' 
FV approximations are fully computable,
and they converge 
to the smooth 
solutions on small times 
at least 
\cite{cances-mathis-seguin-2016,jovanovic-rohde-2006}.
Despite the lack of well-defined global solutions, they usually allow one to numerically explore a hyperbolic system of conservation laws 
in some useful regimes 
thanks to a 
physically-based guarantee of stability\footnote{
 Even though they do not converge in all cases, and some meaningul solutions are not captured.
 In particular, we are aware that 
 such FV approximations are likely to remain consistent only for small times, 
 which is a problem to capture e.g. 
 steady states when $\lambda\ll1$. 
 Asymptotics preserving schemes are however more involved and will be studied later in future works, in a second numerical exploration of our model.
}.

But in our case, \eqref{eq:SVHE} plays the role of \eqref{eq:secondprinc} with $S$ replaced by $HE$ which is not convex, 
so it seems we cannot use 
the standard procedure.

Therefore, after recalling the standard case for the sake of clarity,
we present a modification of the standard FV strategy in Section~\ref{sec:oned},
which relies on a previous analysis of the Lagrangian reformulation as it is usual for Eulerian systems \eqref{eq:quasilinear} like SVM that possess contact-discontinuity waves
(see e.g. \cite{bouchut-2004}).

\subsection{Finite-Volume 
approach 
to standard 
conservation laws} 
\label{sec:general}

Given a tesselation of $\R^2$ using polygonal cells $V_i$ ($i\in\NN$), 
consider first $q_h(t)=\sum_i q_i(t) 1_{V_i}$ a semi-discrete FV 
approximation of $q$ solution to a Cauchy problem for \eqref{eq:quasilinear} 
with $q_h(0)\equiv q_h^0\approx q(0)$ at $t=0\equiv t^0$.
To define $q_h(t)$ at $t>0$,
integration from time $t^n$ to $t^{n+1}=\sum_{k=0}^{n-1} \tau^k$ 
($\tau^k>0$) 
for standard 
systems of conservation laws 
like \eqref{eq:quasilinear}
is usually splitted into two sub-steps as follows.

\smallskip

First, for each 
$n\in\NN$, one integrates the 
flux terms forward 
so 
$q_h^{n+1,-}$ 
approximates the 
solution at $t^{n+1,-}$ to the Cauchy problem with 
$q_h(t^n)$ 
at $t^n$ for
\begin{equation}
\label{eq:quasilinearhomogeneous}
\partial_t q + \partial_i F_i(q) = 0
\end{equation}
on $[t^n,t^{n+1})$. 
Denoting $\bn_{i\to j}$ the unit normal from 
$V_i$ to $V_j$ at 
$\Gamma_{ij}\equiv\overline{V_i}\cap\overline{V_j}$, 
\emph{domain-preserving} 1D Riemann solvers  
allow one to define \emph{admissible} 
FV approximations
$q^{n+1,-}_h=\sum_i q^{n+1,-}_i 1_{V_i}$, i.e. 
in the domain of $q$, by 
the 
formula 
\begin{equation}
\label{eq:conservation}
q_i^{n+1,-} = q_i^n - \tau^n \sum_{ \Gamma_{ij}\equiv\overline{V_i}\cap\overline{V_j} \neq \emptyset} 
\frac{|\Gamma_{ij}|}{|V_i|} \bF_{i\to j}(q_i^n,q_j^n;\bn_{i\to j}) 
\end{equation} 
through numerical fluxes $\bF_{i\to j}(q_i^n,q_j^n;\bn_{i\to j})$ precised below,
under a 
CFL condition on $\tau^n$ (see Prop.~\ref{prop:convex}). Moreover, one gets a \emph{fully admissible} $q_h(t^{n+1})=\sum_i q^{n+1}_i 1_{V_i}$, which is also 
consistent with the dissipation \emph{inequality} associated with 
\eqref{eq:secondprinc} 
in addition to \eqref{eq:quasilinearhomogeneous},
when $\bF_{i\to j}$ use \emph{entropy-consistent} Riemann solvers 
(see Prop.~\ref{prop:decay}).
Second, 
source terms are 
integrated backward (see Prop.~\ref{prop:discretedissipation}): 
\beq
\label{eq:convexcombination}
q^{n+1}_i = \left( q^{n+1,-}_i  + \frac{\tau^n}{\lambda_i} q_i^\infty(q^{n+1}_i) \right)/\left(1+\frac{\tau^n}{\lambda_i}\right) \,.
\eeq

\smallskip

Recall that by Galilean invariance, 
in \eqref {eq:conservation}, 
standard numerical fluxes read 
\begin{equation}
\label{eq:rotationinvariantflux} 
\bF_{i\to j}(q_i,q_j;\bn_{i\to j})=\Oij\tilde\bF(\Oij^{-1}q_i,\Oij^{-1}q_j)
\end{equation}
with 
$\Oij^{-1} q$ 
in a local 
basis 
$(\bn_{i\to j},\bn_{i\to j}^\perp)$ 
rather than in 
$(\be_x,\be_y)$,
and with
\begin{multline}
\label{eq:numflux}
\tilde\bF_{i\to j}(\Oij^{-1}q_i,\Oij^{-1}q_j)
\\
= \Oij^{-1}\bF(q_i)\bn_{i\to j}-\int_{-\infty}^0 \left( R(\xi,\Oij^{-1}q_i,\Oij^{-1}q_j)-\Oij^{-1}q_i\right)d\xi
\end{multline}
defined simply with a 1D 
\emph{Riemann solver} $R(\xi,\tilde q_i,\tilde q_j)$, 
i.e. a well-defined 
solution to the 1D Riemann problem 
for 
\eqref{eq:quasilinearhomogeneous}
with initial condition $\tilde q_i 1_{a<0} + \tilde q_j 1_{a>0}$ on $\R\ni a$,
which is a function of $\xi=a/t$, 
or some conservative approximation 
(termed \emph{simple} in the latter case when $R(\cdot,\tilde q_i,\tilde q_j)$ is piecewise constant). 

Recall also that 
information propagates at finite speed in explicit FV 
approximation $q_h^{n+1,-}$, like in 
\eqref{eq:quasilinearhomogeneous}. 
Moreover, that speed is consistent 
when the maximal speed $s(q_l,q_r)>0$ of the waves in $R(\cdot,q_l,q_r)$ is bounded continuously as a function of $q_l,q_r$.
And $q_h^{n+1,-}$ is \emph{admissible}
if the 
Riemann solver $R(\cdot,q_l,q_r)$ in \eqref{eq:numflux} 
preserves the 
domain of $q$
(i.e. the 
values assumed by 
$\xi\to R(\xi,q_l,q_r)$ 
belong to the admissibility 
domain for $q$)
under CFL condition \cite{bouchut-2004}:
\begin{proposition}
\label{prop:convex}
If a numerical flux is given by \eqref{eq:rotationinvariantflux}, \eqref{eq:numflux} with a 1D Riemann solver $R(\cdot,q_l,q_r)$
that preserves a convex 
domain for $q$ and has bounded maximal 
wavespeed $s(q_l,q_r)>0$, 
then the FV approximate solution \eqref{eq:conservation} to \eqref{eq:quasilinearhomogeneous}
also preserves the 
domain for $q$ 
under the CFL condition \eqref{eq:CFLsoft}
\beq
\label{eq:CFLsoft}
\forall i \quad \tau^n \sum_{j} 
{|\Gamma_{ij}|s(\Oij^{-1}q_i^n,\Oij^{-1}q_j^n) }/{|V_i|} \le 1 \,.
\eeq
\end{proposition}
\begin{proof}
It suffices to rewrite \eqref{eq:conservation} with \eqref{eq:numflux} as
\begin{multline}
\label{eq:cvxcombin1}
q_i^{n+1,-} = q_i^n \left( 1 - \tau^n \sum_j 
{|\Gamma_{ij}|s(\Oij^{-1}q_i^n,\Oij^{-1}q_j^n) }/{|V_i|} \right) 
\\
+ \sum_j (\tau^n
{|\Gamma_{ij}|}/{|V_i|}) \int_{-s(\Oij^{-1}q_i^n,\Oij^{-1}q_j^n) }^0 \Oij R(\xi,\Oij^{-1}q_i^n,\Oij^{-1}q_j^n) d\xi
\end{multline}
i.e. as a convex combination under the CFL condition \eqref{eq:CFLsoft}.
\end{proof}

A direct consequence of the admissibility of $q_h^{n+1,-}$ is that $q_h(t^{n+1})$ computed by 
\eqref{eq:convexcombination}
is also admissible, in the 
domain of $q$. 
Moreover it holds
\beq
\label{eq:entropydecrease}
S(q_h^{n+1})\le S(q_h^{n+1,-}) \le S(q_h^{n})
\eeq
i.e. the (convex) mathematical entropy necessarily decreases.
But this is 
not enough yet for $q_h(t^{n+1})$ to be \emph{fully admissible}, 
i.e. to approximate 
the \emph{inequality} 
\begin{equation}
\label{eq:secondprincineq}
\partial_t S(q) + \div 
\bG(q) \le -D(q)
\end{equation}
as an admissibility criterion 
formulating the thermodynamics second principle. 
%

\smallskip

Next, if 
$R(\cdot,q_l,q_r)$ in 
\eqref{eq:rotationinvariantflux}, \eqref{eq:numflux} is also \emph{entropy-consistent} 
(with \eqref{eq:secondprincineq} in 1D, 
see \eqref{eq:discreteentropy} below) 
then, under a CFL condition \emph{more stringent} than \eqref{eq:CFLsoft},
a discrete 
version of \eqref{eq:secondprincineq} holds 
when $k=0=D$ (Prop.~\ref{prop:decay}), and 
when $k\ge0,D\ge0$ after 
backward integration of the sources (Prop.~\ref{prop:discretedissipation}):
\begin{proposition}
\label{prop:decay}
If the flux \eqref{eq:rotationinvariantflux}, \eqref{eq:numflux} uses a 1D Riemann solver $R(\cdot,q_l,q_r)$
that preserves the 
domain of $q$ and is entropy-consistent with \eqref{eq:secondprincineq} when $k=0=D$ in the sense that, 
given 
admissible 
state-vectors $q_l,q_r$ and a direction $\bn$, a 
discrete entropy-flux vector 
$ \tilde G_{\bn}(q_l,q_r)=-\tilde G_{\bn}(q_r,q_l)$ 
satisfies 
\begin{multline}
\label{solversufficientcondition}
%
\bG(q_r)\cdot\bn + \int\displaylimits_0^{+\infty} \left( S\left(R(\xi,q_l,q_r)\right) -S(q_r) \right)d\xi
\\
\le
\tilde G_{\bn}(q_l,q_r)
\le 
\bG(q_l)\cdot\bn - \int\displaylimits_{-\infty}^0 \left( S\left(R(\xi,q_l,q_r)\right) -S(q_l) \right)d\xi \,,
\end{multline}
then, under the CFL condition \eqref{eq:CFL}
\beq
\label{eq:CFL}
\tau^n s_i^n \sum_j 
{|\Gamma_{ij}|}/{|V_i|} \le 1
\eeq
where $s_i^n := \max_j s(\Oij^{-1}q_i^n,\Oij^{-1}q_j^n)$, the FV approximation \eqref{eq:conservation} 
preserves the 
domain of $q$ 
and satisfies the following 
discrete version of \eqref{eq:secondprincineq} (with $D=0$): 
\beq
\label{eq:discreteentropy}
S(q_i^{n+1,-}) - S(q_i^n) + \tau^n \sum_j \frac{|\Gamma_{ij}|}{|V_i|} \tilde G_{\bn_{i\to j}}(\Oij^{-1}q_i^n,\Oij^{-1}q_j^n)
\le 
0 \,.
\eeq
\end{proposition}

\begin{proof}
We follow the 1D proof in \cite{bouchut-2004} and first rewrite \eqref{eq:conservation} with \eqref{eq:numflux} as
\beq
\label{eq:cvxcombin2}
q_i^{n+1,-} 
= \sum_j \tau^n \frac{|\Gamma_{ij}|}{|V_i|} 
 \int_{-\left(\sum_j \tau^n \frac{|\Gamma_{ij}|}{|V_i|}\right)^{-1}}^0 \Oij R(\xi,\Oij^{-1}q_i^n,\Oij^{-1}q_j^n) d\xi
\eeq
i.e. as a convex combination 
that depends only on the Riemann solver under the stringent CFL condition \eqref{eq:CFL}.
We can now use Jensen inequality with \eqref{eq:cvxcombin2}, and next \eqref{solversufficientcondition} with
$ 
\int\displaylimits_0^{+\infty} S\left(R(\xi,q_l,q_r)\right) = - \int\displaylimits_{-\infty}^0  S\left(R(\xi,q_r,q_l)\right)
$, to get
\begin{multline}
\label{eq:numfluxHLLintegralentropy1}
\int\displaylimits_{-s_i^n }^0 S\left(\Oij R(\xi,\Oij^{-1}q_i^n,\Oij^{-1}q_j^n)\right)d\xi
=
\int\displaylimits_{-s_i^n}^0 S\left(R(\xi,\Oij^{-1}q_i^n,\Oij^{-1}q_j^n)\right)d\xi
\\
\le \left(\tau^n s_i^n\right) \: S(q_i^n)
- \tau^n \left( \tilde G_{\bn_{i\to j}}(\Oij^{-1}q_i^n,\Oij^{-1}q_j^n) - \bG(q_i^n)\cdot\bn_{i\to j} \right)
\end{multline}
(recall 
$S$ 
is a \emph{convex} function of the \emph{rotation-invariants} of the state-vector $q$),
and we finally obtain \eqref{eq:discreteentropy} with $\sum_j |\Gamma_{ij}| 
\bn_{i\to j} = \bzero$.
\end{proof}

\mycomment{ 
In the next Section~\ref{sec:oned}, we build an entropy-consistent 
Riemann solver $R$ for our SVUCM model 
following \cite{bouchut-2003,bouchut-2004}.
}

The bound 
on $q_h(t^{n+1})$ provides one with more 
stability\footnote{ 
  The convex 
  domain for $q$ is indeed 
  preserved as a consequence of Prop.~\ref{prop:convex},
  insofar as \eqref{eq:CFL} is more stringent than \eqref{eq:CFLsoft}.
} than Prop.~\ref{prop:convex}.
In particular, using \eqref{eq:discreteentropy}, Prop.~\ref{prop:decay} provides one with an a priori error estimate for FV approximations of 
smooth solutions 
to the 
conservation laws \eqref{eq:quasilinearhomogeneous}, see e.g. \cite{cances-mathis-seguin-2016}.
%
%
And backward integration \eqref{eq:convexcombination} 
of the source term next provides one with a \emph{fully admissible} approximation $q_h(t^{n+1})$: 
\mycomment{When is the entropy bound sufficient for convergence ? scalar and 1D case}
\begin{proposition}
\label{prop:discretedissipation}
For any $\tau^n>0$, using \eqref{eq:convexcombination} with $q_h^{n+1,-}$ from Prop.~\ref{prop:decay} satisfying \eqref{eq:discreteentropy} 
yields $q_h^{n+1}$ satisfying the following discrete version of \eqref{eq:secondprincineq}:
\begin{multline}
\label{eq:discreteentropy2}
S(q_i^{n+1}) - S(q_i^n) + \tau^n \sum_j \frac{|\Gamma_{ij}|}{|V_i|} \tilde G_{\bn_{i\to j}}(\Oij^{-1}q_i^n,\Oij^{-1}q_j^n)
\\
\le 
-\tau^n k |\bU_i^{n+1}|^2 -\tau^n D(q_i^{n+1})
\end{multline}
\end{proposition}
\begin{proof} To \eqref{eq:discreteentropy}, add \eqref{eq:convexcombination} tested against $\grad_q S(q_i^{n+1})$ for all cells $V_i$ 
(which is possible since $q_i^{n+1}$ is admissible as a convex combination in a convex admissibility domain):
\eqref{eq:discreteentropy2} results from the convexity of $S$ and the definition of $D$. 
\end{proof}

\smallskip

For SVM 
neither Prop.~\ref{prop:decay} nor the simpler consequence \eqref{eq:entropydecrease} of Prop.~\ref{prop:convex} can be straightforwardly used 
because we are not aware of 
conservative variables $q$ such that $HE$ is convex 
on the whole admissible domain, i.e. is a 
mathematical entropy.
%
But note that $HE$ equals $\tilde S = H\tilde E-H\log(|\bA_h|A_{cc})$ 
when $H|\bF|=1$, and it is convex with respect to 
$\tilde q = (H,H\bF,H\bU,H A_{cc}^{1/4})$, and $\bA_h$ which is simply transported. 
So the standard FV approximation procedure above can be used on slightly modifying the time-splitting: 
segregating the time-evolution of $\tilde q$ and of $\bA_h$ in the first split step.
%
We propose such a discretization of SVM in the sequel
that also takes advantage of the possibility to rewrite the SVM system 
in Lagrangian coordinates to strike a balance between entropy stability,
and numerical accuracy (especially for the contact discontinuities satisfied by $\bA_h,A_{cc}$).
Following \cite{bouchut-2003,bouchut-2004}, we construct entropy-consistent Riemann solvers for 
(the $\tilde q$ sub-system of) a relaxed SVM system in Eulerian coordinates 
obtained with the help of a BGK approximation 
in Lagrangian coordinates.
The SVM system in Lagrangian coordinates is studied in Appendix~\ref{app:riemann}.
In the next Sec.~\ref{sec:oned}, we use the results of Appendix~\ref{app:riemann} to discretize SVM \emph{in Eulerian coordinates}.

\subsection{Finite-Volume approach 
to Saint-Venant-Maxwell} 
\label{sec:oned}

We adapt the framework of Section~\ref{sec:general} 
to compute 
discrete FV fields
$$
q = (H ,H F^x_a,H F^y_a,H F^x_b,H F^y_b,H U^x,H U^y,H A_{cc},H A_{aa},A_{ab}/\sqrt{A_{aa}A_{bb}},H A_{bb})
$$
that 
solve SVM 
on a 
space-tesselation (i.e. in Eulerian description, 
with cells paving the same 
space at all time steps).
It splits 
time-integration into 2 steps.

\smallskip

\underline{First step}: Given a FV approximation 
$q_i^n$ 
of $q$ at 
time $t^n$, $n\in\N$, 
we consider 
the homogeneous SVM system without source term.
We use a transport-projection method \cite{bouchut-2004a} 
based on an approximation of (Eulerian) SVM:
\beq
\label{eq:SVUCM0detailsub1}
\begin{aligned}
& \partial_t \tilde H + \partial_j (\tilde H \tilde U^j ) = 0
\\
& \partial_t (\tilde H F^i_\alpha) + \partial_j (\tilde H \tilde U^j F^i_\alpha - \tilde H \tilde F^j_\alpha U^i ) = 0
\\
& \partial_t (\tilde H U^i) + \partial_j (\tilde H \tilde U^j U^i + \tilde H \tilde F^j_\alpha \Pcal^i_\alpha 
) = 0
\\ 
& \partial_t (\tilde H A_{cc}) + \partial_{j} (\tilde H \tilde U^j A_{cc} ) = 0 
\\ 
& \partial_t (\tilde H A_{\alpha\beta}) + \partial_j ( \tilde H \tilde U^j A_{\alpha\beta} ) = 0
\end{aligned}
\eeq
where $\tilde H\approx |\bF_h|^{-1}$, $\tilde\bU\approx\bU$, $\tilde\bF_h\approx\bF_h$ still have to be defined such that 
not only the involution $\tilde H|\bF_h|=1$ of SVM is (approximately) preserved, but also 
\begin{equation}
\label{eq:piolatilde} 
\partial_j( \tilde H \tilde F^j_\alpha ) \approx 0 \quad \forall \alpha \,.
\end{equation}
Indeed, using $d x_i = \tilde U^i dt + \tilde F^i_\alpha d a_\alpha$ and equality \eqref{eq:piolatilde} in the smooth case, one can retrieve from \eqref{eq:SVUCM0detailsub1}
the SVM equations in Lagrangian coordinates $(a_\alpha)$. 
Then, the entropy-stability of a flux-splitting FV scheme for SVM in Lagrangian coordinates 
(see Lemma~\ref{lem:entropylag3} in Appendix~\ref{app:lagrange})
can be transferred to a simple 
Riemann solver in Eulerian coordinates, which also captures well contact discontinuities \cite{bouchut-2003}.
Precisely, for the Lagrangian-to-Eulerian mapping, we complement \eqref{eq:SVUCM0detailsub1} by 
\begin{equation}
\label{stronglyconsistent} 
\partial_t (\tilde H H^{-1} ) + \partial_j (\tilde H \tilde U^j H^{-1} ) - \partial_{j}( \tilde H \tilde F^j_\alpha U^j \sigma_{ij}\sigma_{\alpha\beta}F^i_\beta ) = 0
\end{equation}
and we now propose to use 
$$
\tilde H = H 
\quad 
\tilde U^j = U^j 
\quad
\tilde H \tilde F^j_\alpha = E^j_\alpha 
$$
where $E^j_\alpha \approx H F^j_\alpha$ is the $(\alpha,j)$ entry of the cofactor matrix 
of $\bF_h^{-1}$.
Then, \eqref{eq:SVUCM0detailsub1}, \eqref{stronglyconsistent} can be closed using 
for $E^j_\alpha$ (see e.g. \cite{wagner-2009})
\begin{equation}
\label{eq:cofGevolution}
\partial_t E^j_\alpha 
+ U^k \partial_k E^j_\alpha + E^k_\alpha ( \sigma_{ij}\partial_i ) U^k = 0
\end{equation}
which preserves 
\eqref{eq:piolatilde}. 
It is noteworthy that preserving 
\eqref{eq:piolatilde} "discretely" with  
1D Riemann solvers using \eqref{eq:cofGevolution} %
reduces to capturing a contact discontinuity. 

Now, in this first time-integration step of SVM by a splitting approach,
we can 
update the FV approximation of $\tilde q$ 
by \eqref{eq:conservation} on the one hand,
using 
$F(q_i^n,q_j^n;\bn_{i\to j})$ computed from 
a \emph{fully-admissible} 1D Riemann solver for the system \eqref{eq:SVUCM0detailsub1}, \eqref{stronglyconsistent}, 
\eqref{eq:cofGevolution} 
with a flux that is \emph{discontinuous} through contact waves for $\bA_h$.
On the other hand, $\bA_h$ can be transported 
and projected in a segregated sub-step preserving full-admissibility of the solution
($\bA_h$ remains SDP with the 
upwind scheme, and 
$E$ decays as a single convex function in $\bA_h$). 

Precisely, 
we compute $F(q_i^n,q_j^n;\bn_{i\to j})$ 
with the following simple 1D Riemann solver 
in direction $\be_m = \bn_{i\to j}$ \footnote{In \eqref{eq:SVUCM0detaileul1Drelaxedter}, $m$ is fixed: Einstein convention is not used.}
for 
\eqref{eq:SVUCM0detailsub1}, \eqref{stronglyconsistent}, \eqref{eq:cofGevolution}
(motivated by a flux-splitting for SVM in Lagrangian coordinates, see App.~\ref{app:lagrange}): 
\beq
\label{eq:SVUCM0detaileul1Drelaxedter}
\begin{aligned}
& \partial_t H  + \partial_m ( H U^m ) = 0
\\
& \partial_t ( H F^\delta_e ) + \partial_m ( H U^m F^\delta_e - E^m_e U^\delta ) = 0 
\\
& \partial_t ( H F^\delta_f ) + \partial_m ( H U^m F^\delta_f ) = 0 
\\
& \partial_t ( H U^\delta ) + \partial_m ( H U^m U^\delta + E^m_e \Pi^\delta_e ) = 0 
\\ 
& \partial_t ( H \Pi^\delta_e / c^2 ) + \partial_m ( H U^m \Pi^\delta_e / c^2 + E^m_e U^\delta ) = 0 
\\
& \partial_t ( H \Vcal_e ) + \partial_m ( H U^m \Vcal_e + E^m_e \Zcal_{ee} ) = 0 
\\ 
& \partial_t ( H \Zcal_{ee}/ c^2  ) + \partial_m ( H U^m \Zcal_{ee} / c^2  + E^m_e \Vcal_{e}^\delta ) = 0 
\\ 
& \partial_t ( H c^2 )  + \partial_m ( H U^m c^2 ) = 0
\\ 
& \partial_t ( H E^m_e )  + \partial_m ( H U^m E^m_e ) = 0
\end{aligned}
\eeq 
where $\delta\in\{\parallel,\perp\}$ denotes the two components in a Cartesian basis $(\bn^\parallel,\bn^\perp)$ for the geometric coordinates (Eulerian description),
and $(\be_e,\be_f)$ is a basis for the material coordinates (Lagrangian description) 
yet to be precised.

Our motivation for 
\eqref{eq:SVUCM0detaileul1Drelaxedter} is the possibility to reformulate it as a mapping to Eulerian coordinates 
of the \emph{fully-admissible} 
1D Riemann solver constructed in App.~\ref{app:lagrange} 
for SVM in Lagrangian coordinates:
\begin{proposition} 
\label{prop:reformulation}
If $\be_m=\bn^\parallel$, $F^\parallel_f\equiv 0$, $E^m_eF^\perp_f\equiv 1$,
then 
\eqref{eq:SVUCM0detaileul1Drelaxedter} 
also writes
\beq
\label{eq:SVUCM0detaileul1Drelaxedterdiagonal}
\begin{aligned}
& 0 = \partial_t F^\delta_f 
+ U_m \partial_m F^\delta_f
\\
& 0 = \partial_t \left( F^\delta_e + {\Pi^\delta_e}/{c^2} \right) 
+ U^m \partial_m\left( F^\delta_e + {\Pi^\delta_e}/{c^2} \right)
\\
& 0 = \partial_t \left( U^\delta \pm {\Pi^\delta_e}/c \right) 
+ \left( U^m \pm c H^{-1} E^m_e\right) 
\partial_m \left( U^\delta \pm {\Pi^\delta_e}/c \right) 
\\
& \quad + c H^{-1} \left( U^\delta \pm {\Pi^\delta_e}/c \right) \partial_m ( c E^m_e )
\\
& 0 = \partial_t \left( \Vcal_e \pm {\Zcal_{ee}}/{c} \right)
+ \left( U^m \pm c H^{-1} E^m_e\right) 
\partial_m \left( \Vcal_e \pm {\Zcal_{ee}}/{c} \right) 
\\
& \quad + c H^{-1} \left( \Vcal_{e} \pm {\Zcal_{ee}}/c \right) \partial_m ( c E^m_e )
\\
& \partial_t \left( H^{-1} + \Zcal_{ee} / c^2 \right)
+ U^m \partial_m \left( H^{-1} + \Zcal_{ee} / c^2 \right) = 0
\end{aligned}
\eeq 
for any initial condition such that
$
\Zcal_{ee} := \Pi^\parallel_e F^\perp_f - \Pi^\perp_e F^\parallel_f =  \Pi^\parallel_e (E^m_e)^{-1}
$.
\end{proposition}
\begin{proof}
One obtains \eqref{eq:SVUCM0detaileul1Drelaxedterdiagonal} from \eqref{eq:SVUCM0detaileul1Drelaxedter} by direct computation
and for the equivalence it suffices to see that, under assumptions of  Prop.~\ref{prop:reformulation}, it holds:
$$
 U^m \pm c H^{-1} E^m_e = U^m \mp \Pi^m_e/c  \pm c ( H^{-1} + \Zcal_{ee} / c^2 ) E^m_e \,.
$$
\end{proof}

In particular, the Lagrangian eigenstructure is preserved by the mapping:
\eqref{eq:SVUCM0detaileul1Drelaxedterdiagonal} shows that the system 
has 3 linearly degenerate waves with speed
\begin{multline}
\lambda_- = \left( U^m - c H^{-1} E^m_e 
\right)_l
= \left( U^m + \frac{\Pi^m_e}c - c ( H^{-1} + \frac{\Zcal_{ee}}{c^2} ) E^m_e\right)_l^*
\\
\lambda_+ = \left( U^m + c H^{-1} E^m_e 
\right)_r
= \left( U^m - \frac{\Pi^m_e}c + c ( H^{-1} - \frac{\Zcal_{ee}}{c^2} ) E^m_e\right)_r^*
\\
\lambda_0 = (U^m)_l^{*} = (U^m)_r^{*}
\end{multline}
that are 
ordered $\lambda_- \le \lambda_0  \le \lambda_+$ 
if 
$E^m_e\ge 0$. 
Moreover, 
the 3-wave solutions to \eqref{eq:SVUCM0detaileul1Drelaxedter} 
that are initialized at $t^n$ with left/right values in $V_i$/$V_j$ 
such that
\begin{equation}
\label{eq:initialization}
\Pi^\delta_e = \Pcal^\delta_e 
\qquad
\Vcal_e = U^\parallel F^\perp_f 
\qquad
\Zcal_{ee} = \Pcal^\parallel_e F^\perp_f 
\qquad
E^m_eF^\perp_f\equiv 1
\end{equation}
are (formally) consistent with 1D SVM solutions 
provided 
\begin{itemize}
 \item[i)] choosing $\be_f$ such that $F^\parallel_f\equiv 0$ is consistent (i.e. $F^\parallel_f\approx 0$ in reality), and
 \item[ii)] the following evolution equations hold (in the smooth case for some $c^2>0$)
$$
\partial_t ( H \Pcal^\delta_e ) + \partial_m ( H U^m \Pcal^\delta_e  + HF^m_e c^2 U^\delta 
) = 0 \quad \delta \in\{\parallel,\perp\}
$$
\end{itemize}
recalling the analysis in App.~\ref{app:lagrange} 
for SVM in Lagrangian coordinates (compare 
with the 
equation \eqref{eq:pressureeq} satisfied by the true 
$\Pcal^\delta_e$ in Lagrangian description).

\mycomment{There is a consistency error here to be improved by a 5wave solver}

Of course, in practice, 
given any two neighbour cells $V_i$/$V_j$ and $\be_m=\bn_{i\to j}=\bn^\parallel$,
there is no reason why there should exists one 
direction $\be_f$ such that $F^\parallel_f\equiv 0$ in $V_i$/$V_j$.
However, recall that the meaning of the \emph{tensor} conservative variable $H\bF_h$ in SVM 
is not purely physical. 
For the 
Riemann solver, one can therefore use a reconstruction $F^\delta_f \neq H^{-1} (HF^i_\alpha) n^\delta_i f_\alpha $
where $n^\delta_i$, $f_\alpha$ are respectively the coordinates 
of $\bn^\delta$ and $\be_f$ in $(\be_x,\be_y)$ and $(\be_a,\be_b)$
provided $F^\delta_f$ retains the physics: $F^\delta_f$ are the current coordinates in $(\be_x,\be_y)$ (Eulerian description)
of a \emph{material vector} $\be_f$ attached to the reference configuration (useful in Lagrangian description).
And for similar reasons, one can also ``project'' $H\bF_h$ 
at the end of the transport-projection method used in the present first step of our splitting scheme to make cell-values compatible with mass and energy conservation.

Thus, at the beginning of the transport-projection method, we reconstruct $F^\delta_{e,f}$ at each interface such that 
the mass is conserved on each side of the interface,
such that the elastic energy lost by enforcing $F^\parallel_f = 0$ is minimal, 
$F^\parallel_e \ge 0$ 
and the 
error due to enforcing $F^\parallel_f = 0$ is small.
%
%
For consistency and stability reasons, the initial values $c_{l/r}$ should also be well-chosen in the Riemann solver \eqref{eq:SVUCM0detaileul1Drelaxedter} 
see our analysis for SVM in Lagrangian coordinates in App.~\ref{app:lagrange}.
\begin{proposition}
\label{prop:stab}
Given $\be_m=\bn_{i\to j}=\bn^\parallel$ at some interface in between two 
cells $V_i/V_j$, define $(\be_e,\be_f)$
such that $\be_f$ is the arithmetic mean of two right eigenvectors of $\bF|_{V_i},\bF|_{V_j}$ 
with singular value closest to 1 and $F^\parallel_e \ge 0$.
\\
Next, reconstruct $F^\parallel_e = (\lambda H)^{-1}$, $(F^\parallel_f,F^\perp_f) = (0,\lambda )$, $F^\perp_e $ in $V_i/V_j$ such that
\beq
\label{eq:Fperpe}
%
A_{ee} (\lambda H)^{-2} + A_{ff} \lambda^2 + F^\perp_e (2 A_{ef}) \lambda  + A_{ee} (F^\perp_e)^2 = F^i_\alpha A_{\alpha\beta} F^i_\beta =: \Ecal
\eeq
where $A_{ee},A_{ef},A_{ff}$ are the coefficients of the tensor $\bA_h$ in the basis $(\be_e,\be_f)$, 
and $\lambda\ge0$ is chosen close to 1 
in $ \left|\lambda^2 - \frac{A_{ee}\Ecal}{2|\bA|}\right| \le \frac{ A_{ee} \sqrt{\Ecal^2-4|\bA|/H^2}}{2|\bA|} $, hence
$$
F^\perp_e = \left( - A_{ef}\lambda \pm \sqrt{ A_{ef}^2 \lambda^2 + A_{ee} ( \Ecal 
- A_{ee} (\lambda H)^{-2}- A_{ff}\lambda^2) } \right)/A_{ee}
$$
is a real solution 
to \eqref{eq:Fperpe}.
(We choose $x-\log x-1$ as distance to 1 for $x>0$.)

Then, the 1D Riemann solver \eqref{eq:SVUCM0detaileul1Drelaxedter} 
initialized 
with \eqref{eq:initialization} and Lemma \ref{lem:c} for $c$ is fully-admissible in the sense of Prop.~\ref{prop:decay},
for $\tilde S$ 
and any direction $\bn=\bn_{i\to j}$ with flux $\tilde G_{\bn}=(H \tilde S \Vcal_e + \Pi^\delta_e U^\delta)|_{\Gamma_{ij}}$
\emph{provided $H=|\bF_h|^{-1}$} in $V_i/V_j$.
\end{proposition}

\begin{proof}
First, note that the reconstruction using $(\be_e,\be_f)$ always exists. 
In particular, $\Ecal^2\ge 4|\bA|/H^2$ holds when $H = |\bF_h|^{-1}$. Indeed, the inequality rewrites
$$
|\tr(\bF_h^T\bA\bF_h)|^2 \ge 4 |\bF_h^T\bA\bF_h|
$$
which obviously holds for any symmetric positive matrix $\bF_h^T\bA\bF_h$. 

Second, note the solution has exactly the same structure with same intermediate states as in Lagrangian coordinates.
Then, admissibilty $H>0$ as well as the discrete entropy inequality 
(thus, full-admissibility) 
follow from the full-admissibility of the Riemann solver in Lagrangian coordinates, 
provided it is \emph{well-initialized} for the equivalence to hold.
%
\end{proof}

To ensure $H=|\bF_h|^{-1}$ in each cell $V_i$, we propose the following ``projection'' at the end of the transport-projection method.
We modify the singular values $\lambda_{a'},\lambda_{b'}$ in the SVD decomposition of $\bF_h = U^T \mathop{\rm diag}(\lambda) V$ such that in each cell,
the mass $(\lambda_{a'}\lambda_{b'})^{-1} =|\bF_h|^{-1}=H$ is conserved 
and the loss $|\lambda_{a'} V_{a'\alpha}A_{\alpha\beta} V_{a'\beta} + \lambda_{b'} V_{b'\alpha}A_{\alpha\beta} V_{b'\beta} - \tr(\bF_h\bA\bF^T_h)|$
of elastic energy 
is minimized. 
When there exist more than one solution, we choose the one closest to $\bF_h$ before projection in ``energy norm'' $\tr(\bF_h\bA\bF_h^T)$.
Of course, if the transport-projection method for the 2D FV discretization preserved $H=|\bF_h|^{-1}$ we would not need that additional step.
But this preservation is a well-known difficulty in the discretization of hyperbolic systems with involutions. 


\smallskip


In the other sub-step $\bA_h$ can be updated as a solution to transport equations
with a scheme preserving the convex domain of $A_{aa}>0,A_{bb}>0,A_{ab}/\sqrt{A_{aa}A_{bb}}\in(-1,1)$.
As a matter of fact, if we use an upwind scheme with $\bU$ given at interfaces by the Riemann problems of the first sub-step,
which is consistent with the decrease of free-energy in each cell thanks to the convexity of $\tilde S$ with respect to $A_{aa},A_{bb},A_{ab}/\sqrt{A_{aa}A_{bb}}$,
then the second sub-step can be done at the same time as the first one.


\smallskip

\underline{Second step}: Source terms can be integrated as usual with a backward formula in a second time-splitting step.

\section{Conclusion} %
\label{sec:num}

To conclude, we show some numerical simulations illustrating our new model after 
FV discretization, we discuss the results and we list some perspectives.

%

\begin{figure}[htb]\hspace{-2cm}
\begin{tabular}{ll}
\includegraphics[scale=.4]{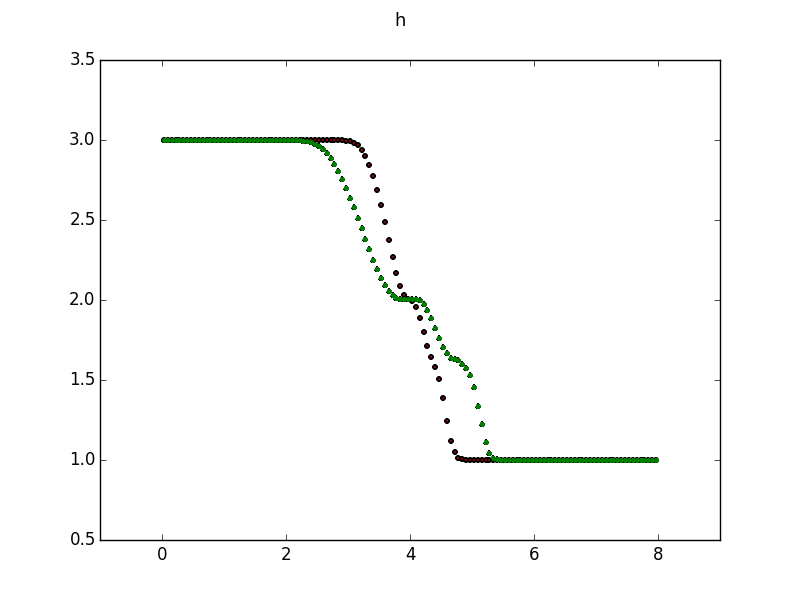} & \includegraphics[scale=.4]{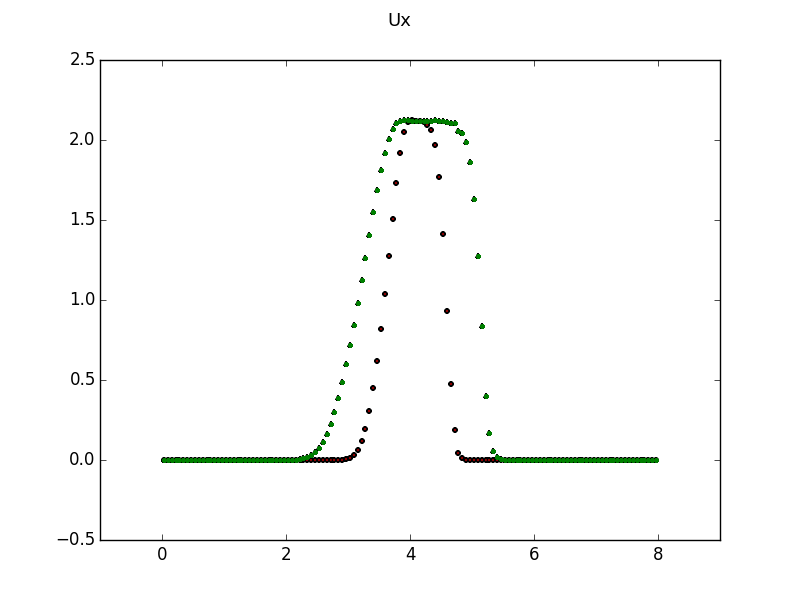}
\\
\includegraphics[scale=.4]{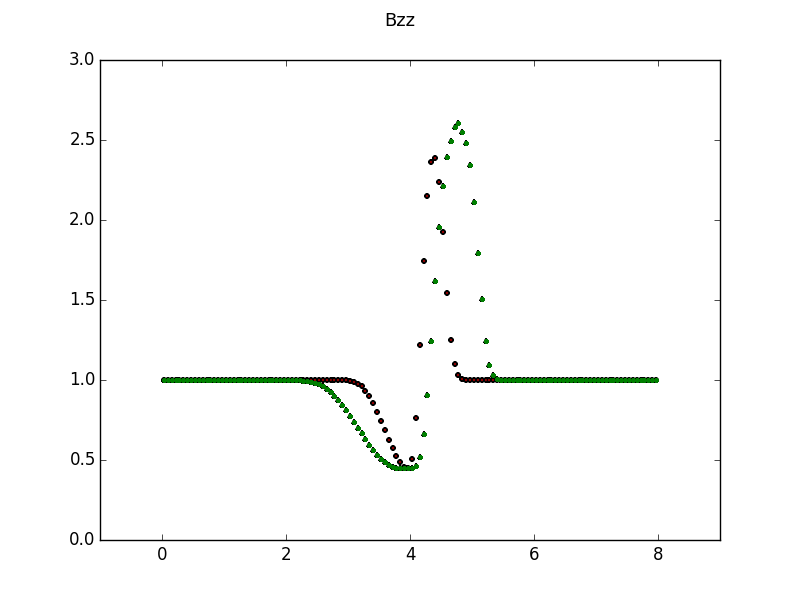} & \includegraphics[scale=.4]{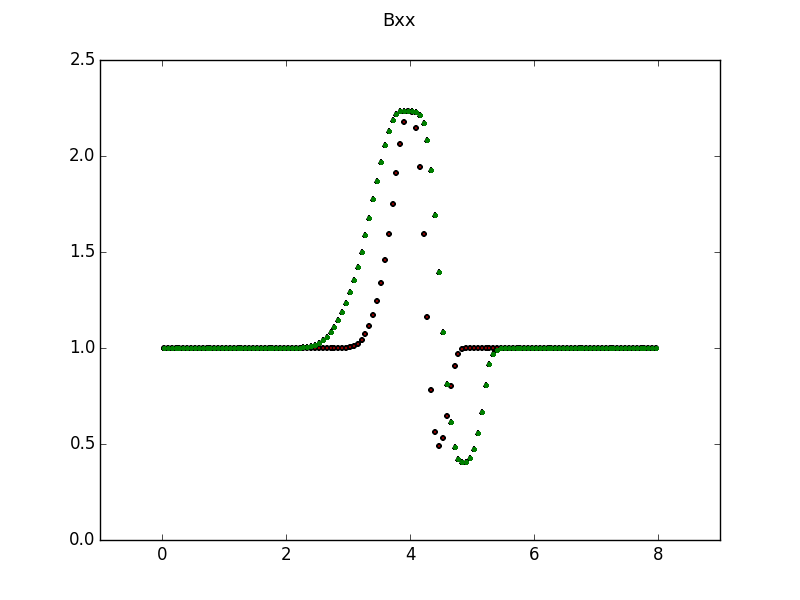}
\end{tabular}
\caption{\label{fig:SVM3_case1_G1}
Test case 1: 1D slice of the variables $H,U^x,B_{zz},B_{xx}$ (from top to bottom, left to right) at times $t=.1$ and $t=.2$ along $x\in[0,8]$ 
}
\end{figure}

\paragraph{Test case 1}
First, recall that our model contains the 1D model SVUCM that was derived in \cite{bouchut-boyaval-2013}:
SVUCM is a closed subsystem of our new model SVM, that admits solutions preserving the translation-invariance of initial conditions.
Moreover, our 1D Riemann solver is also similar to the one constructed in \cite{bouchut-boyaval-2013}.
And our face and cell reconstructions of the deformation gradient variable $\bF_h$ a priori
preserve the translation-invariance of an initial condition on a 1D-conforming grid:
$HF^x_a = 1 = F^y_b$, $F^x_b = 0 = F^y_a$ is preserved by time-evolution on a Cartesian grid $\be_x,\be_y$.

But our FV discretization of SVM does not use exactly the same variables as the one for SVUCM in \cite{bouchut-boyaval-2013}:
$HA_{aa}$, $HA_{cc}$ is used here, 
while $HB^{xx} = HF^x_aA_{aa}F^x_a \equiv A_{aa}H^{-1}$, $HB^{zz} = HF^z_cA_{cc}F^z_c \equiv A_{cc}H^3$ was used in \cite{bouchut-boyaval-2013}.

Now, the 2D numerical results obtained here in the translation-invariant case $HF^x_a = 1 = F^y_b$, $F^x_b = 0 = F^y_a$, $A_{ab} = 0$, $A_{bb} = 1$
with a $2^7\times2^7 = 128\times 128$ Cartesian grid for $(x,y)\in[0,8]^2$, $t\in[0,.2)$ and an initial condition 
$$
H = \begin{cases}
3 & ;\ x<4
\\
1 & ;\ x>4
\end{cases}
$$
at rest $U^x = 0 = U^y$, $A_{aa} = H^2 = A_{cc}^{-1}$ with $G=1$, $\lambda=.1$, $g=10$
compare well with the 1D results in \cite[section 5.5]{bouchut-boyaval-2013} (Test Case 1). 
Without source term, the 1D solution consists exactly in a left-going rarefaction wave, a right-going shock wave,
and a contact-discontinuity wave \cite{boyaval-2018}. 
Here, $\lambda$ is quite large in comparison with $T=.2$ and this is as well-captured in Fig.~\ref{fig:SVM3_case1_G1} as in \cite{bouchut-boyaval-2013}.
Note however that the latter translation-invariant 2D solution is by no way the unique,
and we have indeed observed that other solutions could be captured depending on 
the cell reconstruction at the end of the transport-projection method.


\begin{figure}
\hspace{-2cm}
\begin{tabular}{ll}
\includegraphics[scale=.37]{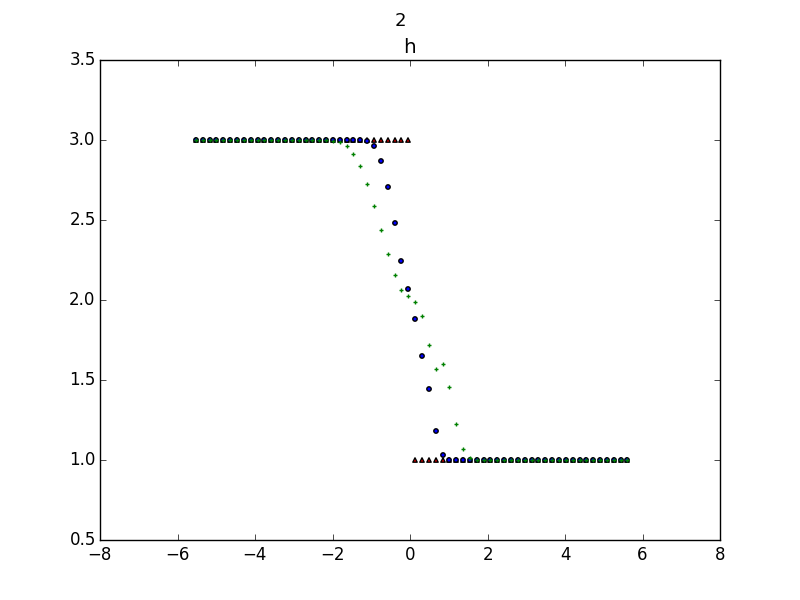} & \includegraphics[scale=.37]{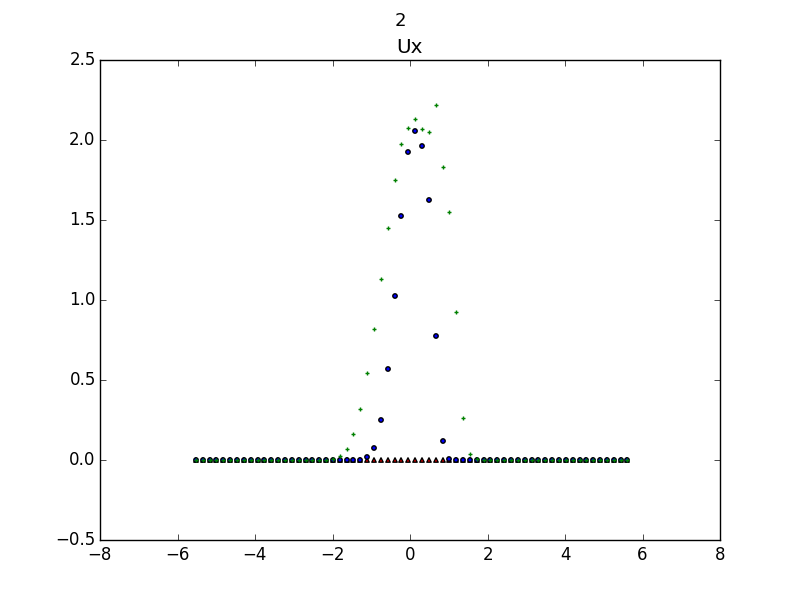}
\\
\includegraphics[scale=.37]{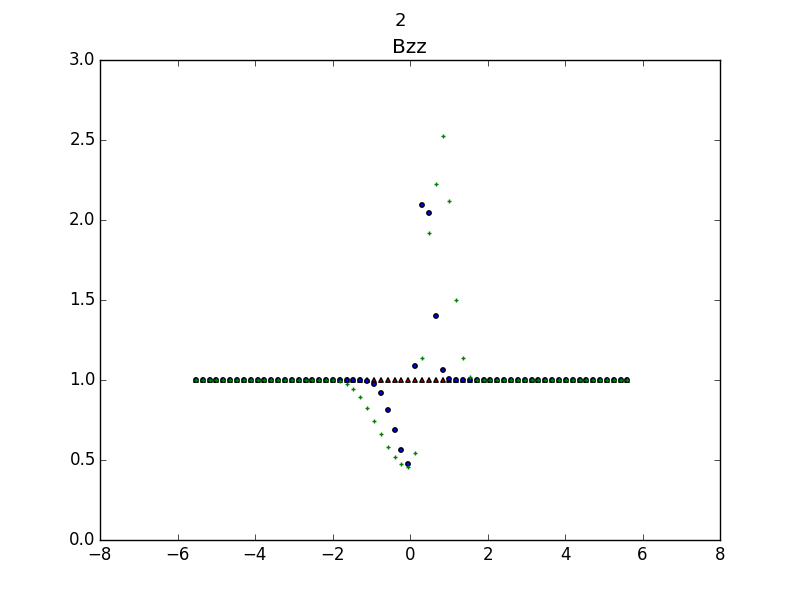} & \includegraphics[scale=.37]{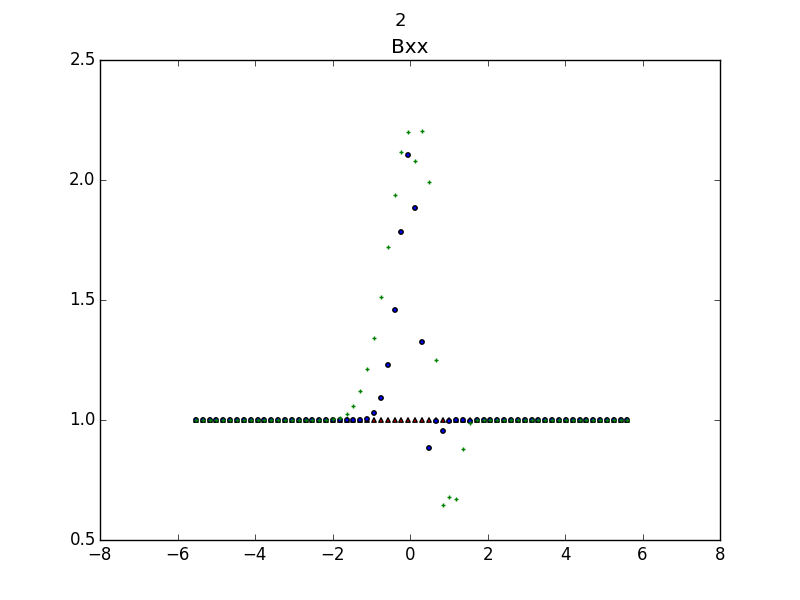}
\end{tabular}
\caption{\label{fig:SVM3_case2_G1}
Test case 2: 1D slice of the variables $H,U^x,B_{zz},B_{xx}$ (from top to bottom, left to right) at times $t=.1$ and $t=.2$ along $x\in[0,8]$ 
}
\end{figure}

\begin{figure}
\hspace{-2cm}
\includegraphics[scale=.4]{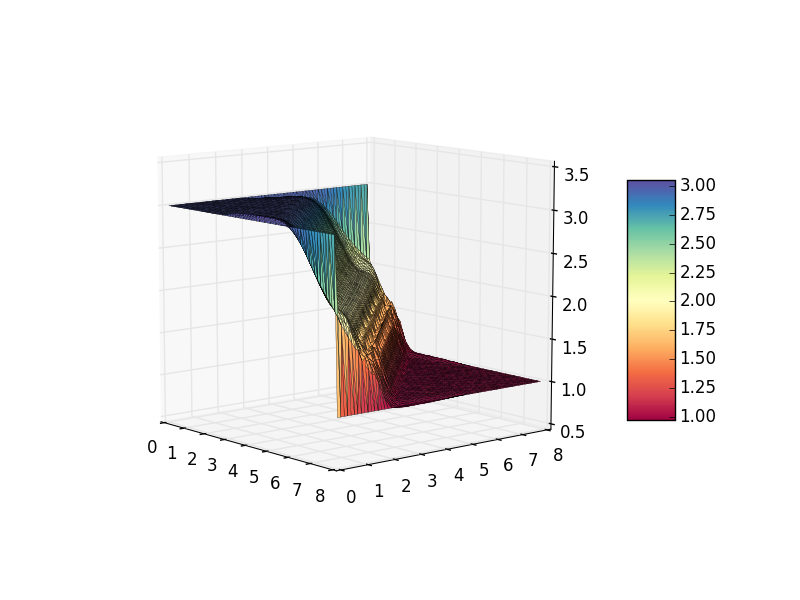}~\includegraphics[scale=.37]{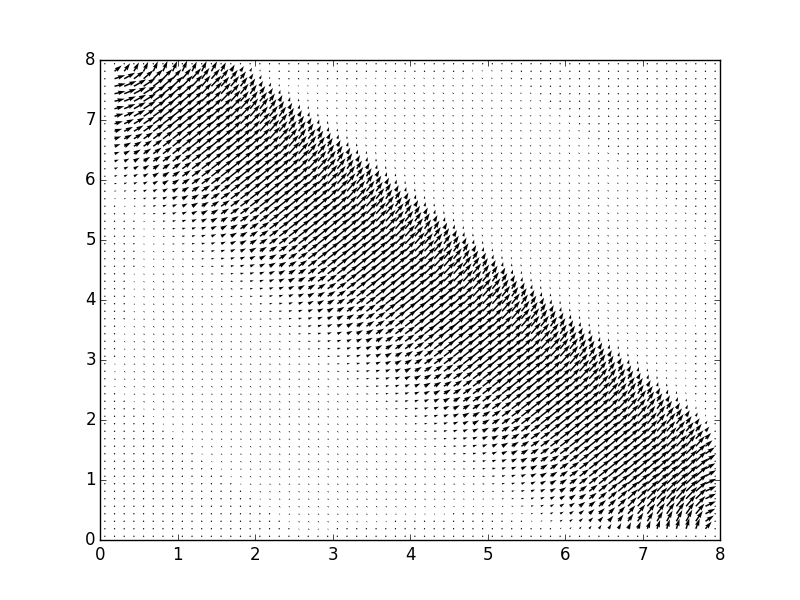}
\caption{\label{fig:SVM3_case2_G1_3D}
Test case 2: 3D view of $H$ (left) and 2D vector field $(U^x,U^y)$ (right) at time $t=.2$  
}
\end{figure}

\begin{figure}
\hspace{-2cm}
\begin{tabular}{ll}
\includegraphics[scale=.37]{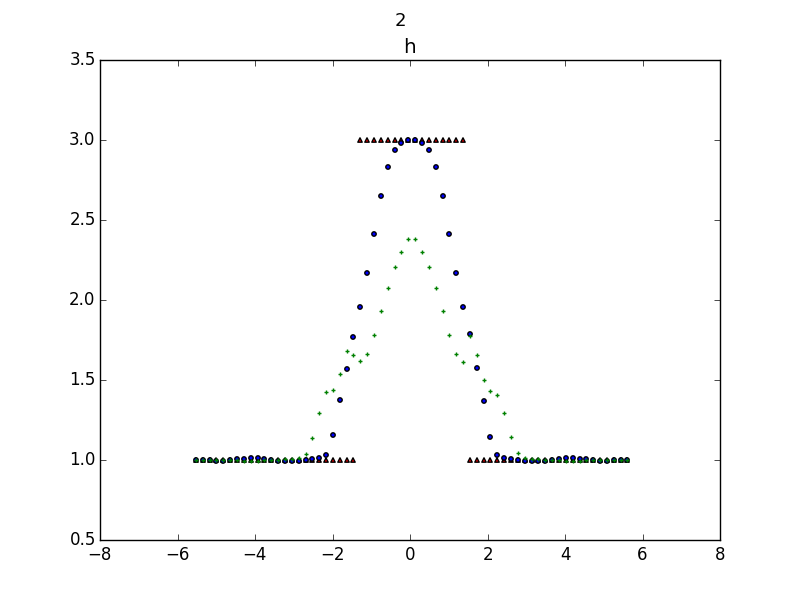} & \includegraphics[scale=.37]{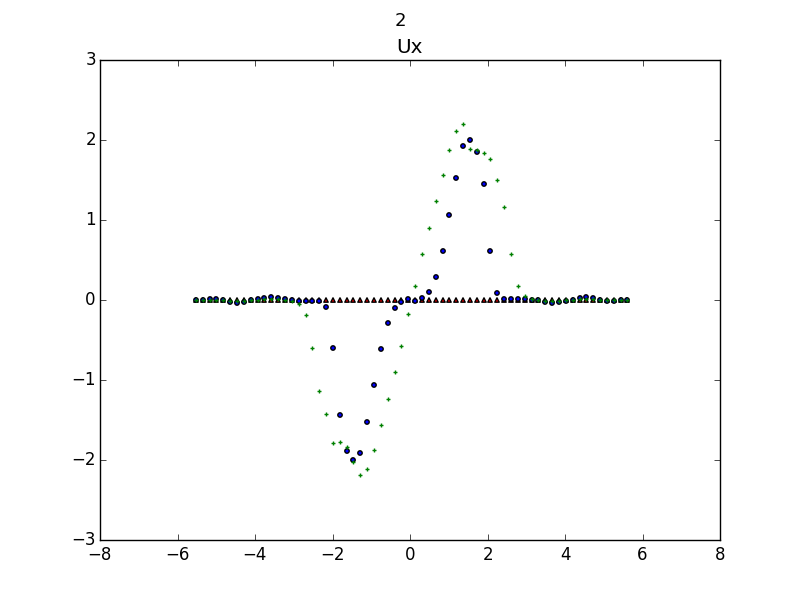}
\\
\includegraphics[scale=.37]{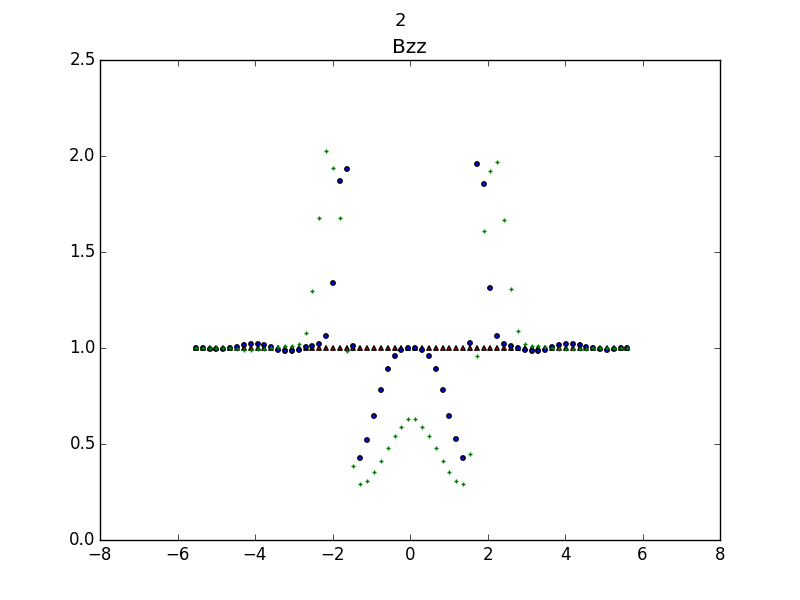} & \includegraphics[scale=.37]{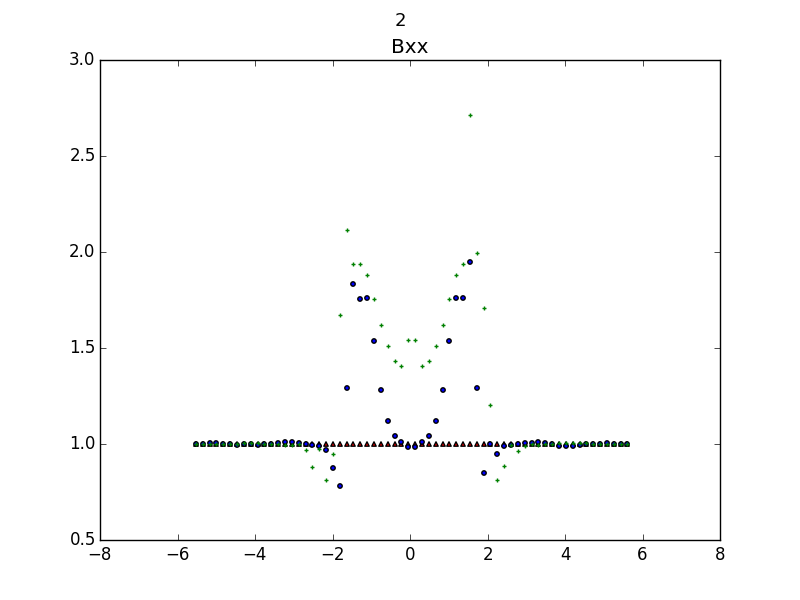}
\end{tabular}
\caption{\label{fig:SVM3_case3_G1}
Test case 3: 1D slice of the variables $H,U^x,B_{zz},B_{xx}$ (from top to bottom, left to right) at times $t=.1$ and $t=.2$ along a radial direction ($\theta$ fixed).
}
\end{figure}

\begin{figure}
\hspace{-2cm}
\includegraphics[scale=.4]{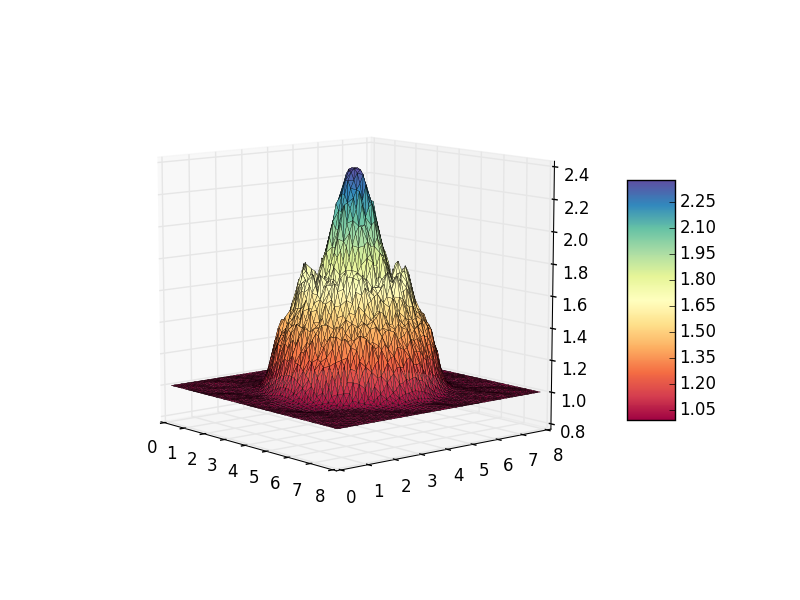}~\includegraphics[scale=.37]{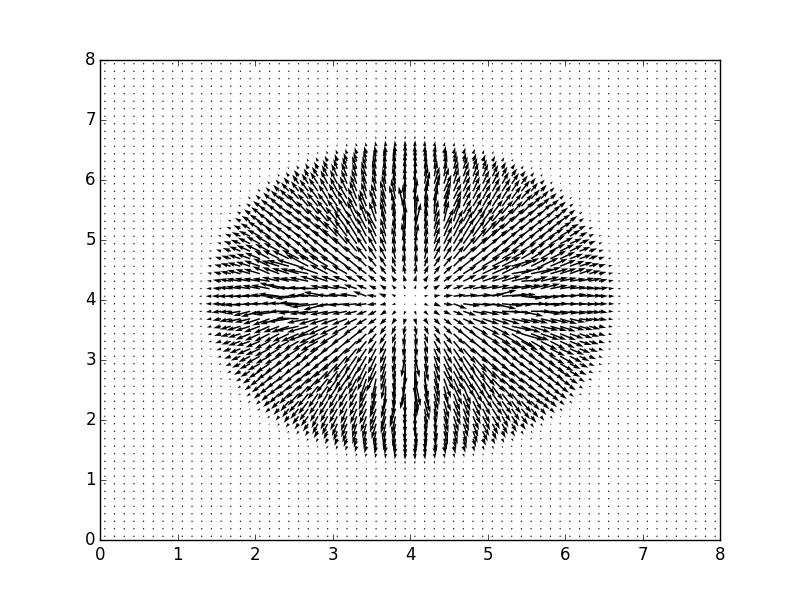}
\caption{\label{fig:SVM3_case3_G1_3D}
Test case 3: 3D view of $H$ (left) and 2D vector field $(U^x,U^y)$ (right) at time $t=.2$  
}
\end{figure}

\paragraph{Test case 2}
Second, we now run the previous test case with the initial 
condition rotated by $\pi/4$ on the same 2D grid (see view in Fig.~\ref{fig:SVM3_case2_G1_3D}).

Although the results in Fig.~\ref{fig:SVM3_case2_G1} compare with the previous simulations, one now sees that translation invariance is broken, see Fig.~\ref{fig:SVM3_case2_G1_3D}.
We believe it is mainly because of our mass-conforming projection of $\bF_h$ in each cell at the end of the transport-projection method,
whose implementation through re-balancing singular values in SVD decomposition does not preserve well the symmetry.
One consequence is that the contact discontinuity is smeared.

\paragraph{Test case 3}
We next simulate 
an axisymetric test case (with rotation invariance): the collapse of a 2D column initially at rest.
Initially, we choose:
$$
H = \begin{cases}
3 & r < 1
\\
1  & r> 1
\end{cases}
$$
in polar coordinates $(\be_r,\be_\theta)$, with $\bF = H^{-1} \be_r\otimes\be_r  + \be_\theta\otimes\be_\theta$.

The results of Fig.~\ref{fig:SVM3_case3_G1} are consistent with an 
axisymetric (rotation-invariant) solution and compares with the unidirectional (translation-invariant) solutions of cases 1 and 2.
Note however that the discrete solution 
loses some symmetry ($\pi/2$-rotation) and the contact discontinuity wiggles, see Fig.~\ref{fig:SVM3_case3_G1_3D}.

\begin{figure}
\hspace{-2cm}
\includegraphics[scale=.38]{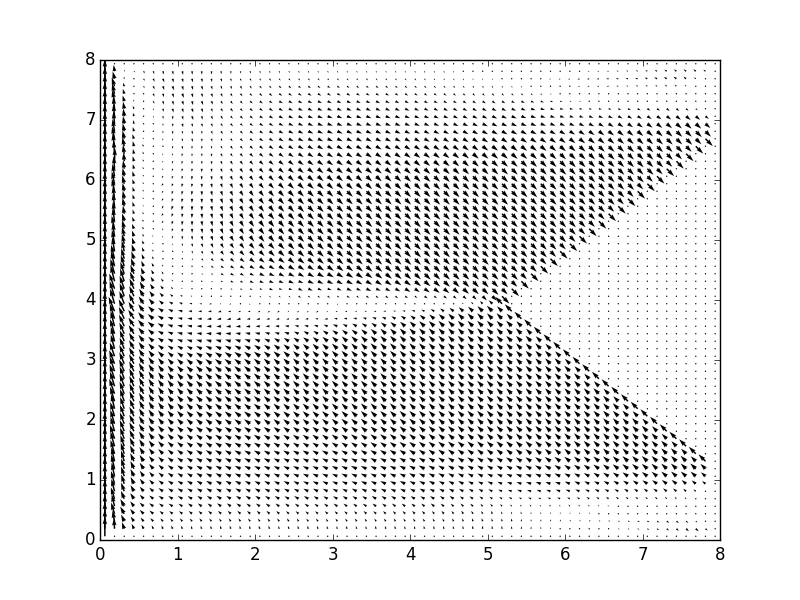}~\includegraphics[scale=.38]{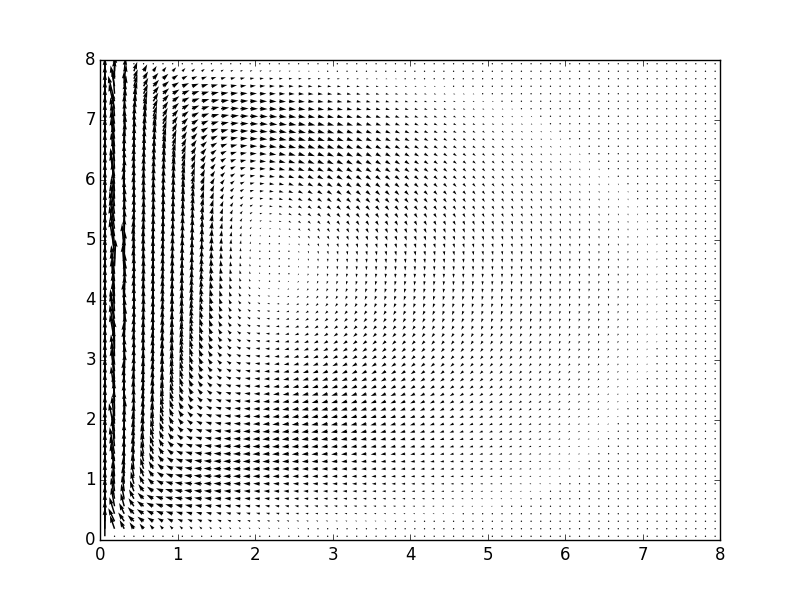}
\caption{\label{fig:SVM3_case4_G1_3D}
Test case 4: 2D vector field $(U^x,U^y)$ at times $t=.5$ and $10$ (left/right).  
}
\end{figure}

\begin{figure}
\hspace{-2cm}
\begin{tabular}{ll}
\includegraphics[scale=.35]{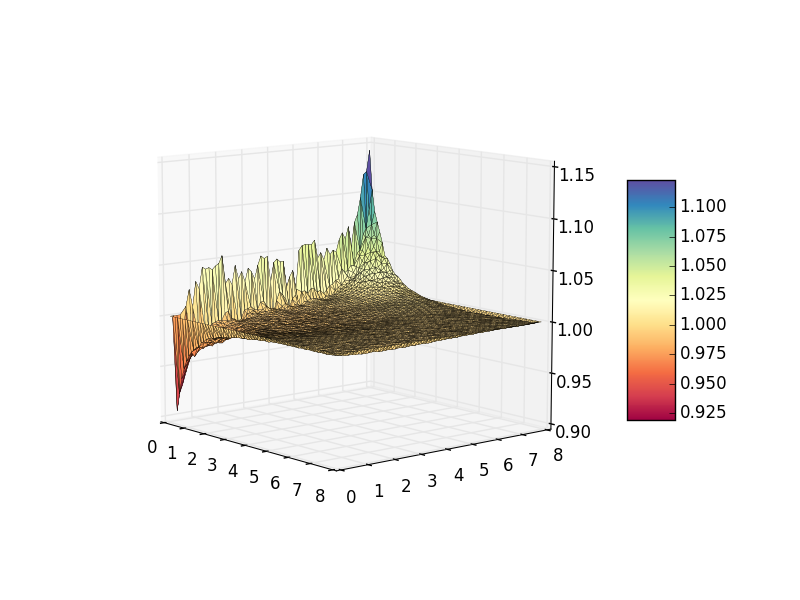} & \includegraphics[scale=.37]{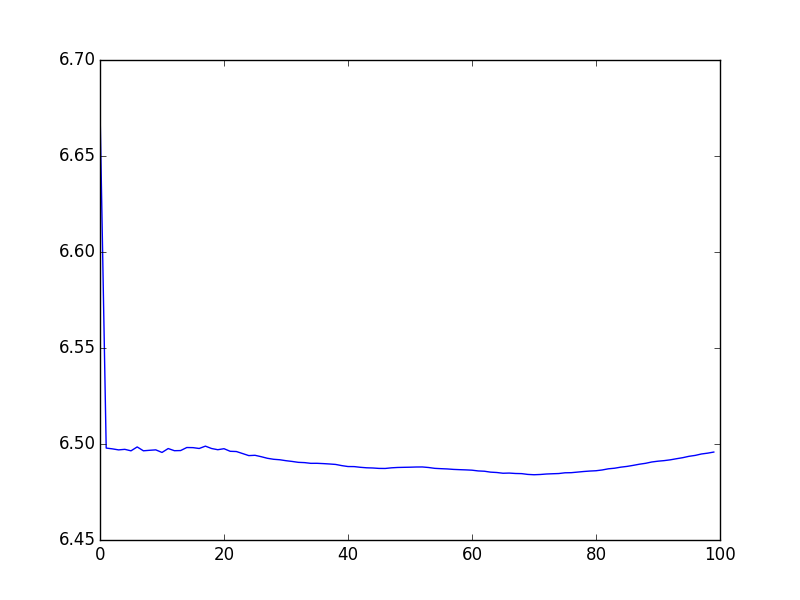}
\\
\includegraphics[scale=.37]{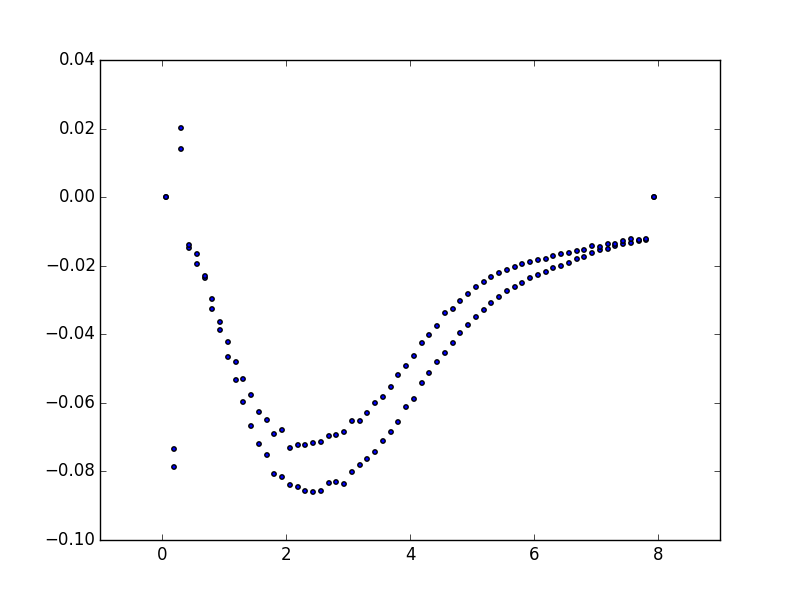} & \includegraphics[scale=.37]{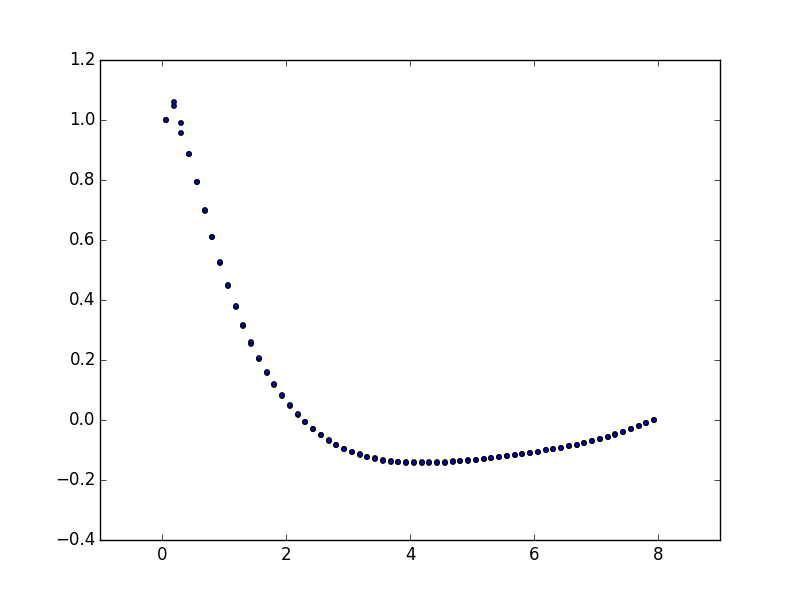}
\end{tabular}
\caption{\label{fig:SVM3_case4_G1}
Test case 4: 3D view of $H$ (top left), time evolution of the space-averaged energy (top right),
and 1D slice of the final-time velocity components $U^x,U^y$ (bottom left and right) along $x$ before and after the centerline $y=4$
(hence the two curves for $U^x$).
}
\end{figure}

\paragraph{Test case 4}
Last, we simulate a standard permanently-sheared flow on a long time range $t\in[0,10]$, 
expecting a steady viscous behaviour asymptotically.
Precisely, we consider a lid-driven cavity on the same 2D grid as in cases 1,2 and 3,
with a uniform initial condition 
$\bF_h=\bI$ at rest, 
$G=1$, $g=10$ and $\lambda=.1$.
The following conditions are 
prescribed at boundaries\footnote{
 We have not precisely studied the 2D Initial-Boundary Value Problem and we impose steady values at boundary cells for the sake of simplicity.
 Note: boundary conditions were not meaningful in cases 1,2 and 3 so long as the front is far enough from the boundary.
}:
$$
\begin{cases}
(H,F^x_a,F^ya,F^x_b,F^y_b,U^x,U^y) = (1,1,1,0,1,0,1) & x=0  
\\
(H,F^x_a,F^ya,F^x_b,F^y_b,U^x,U^y) = (1,1,0,0,1,0,0) & x=8=y\,,\ y=0  
\end{cases}
$$
with $\bA$ at equilibrium (so relaxation source term is zero).

We recall the physical interpretation for the coefficients $F^i_\alpha$.
At time $t$, a basis of ``material vectors'' $\be_a,\be_b$ in the reference configuration (Lagrangian description) has been stretched and turned 
into two ``geometrical vectors'' $\bF_h\be_a,\bF_h\be_b$ resp. of the current configuration (Eulerian description).
So our condition on the ``left'' boundary (the cavity lid) actually means that $\be_a$ (supposedly aligned with $\be_x$ in the reference configuration)
is permanently 
stretched and 
turned by shearing. 

Clearly, standard inviscid Saint-Venant equations would not sustain a developped vortex in such conditions,
while the standard viscous Saint-Venant would immediately develop a developped vortex in the whole cavity after start.
On the contrary, our model can capture the transient development of a vortex until reaching a stationary state, see Fig.~\ref{fig:SVM3_case4_G1_3D}.

With a time-step that remains approximately constant equal to $.007$ (under our $1/2$ CFL condition, recall section~\ref{sec:general}),
our FV discretization quickly reaches a nearly stationary state with $H$ (thus, pressure) almost uniform in space except at the boundary, 
close to which one can observe the stagnation point of a (nearly stationary) vortex, see Fig.~\ref{fig:SVM3_case4_G1}.

Last, note that in each case we have run our simulations on finer and finer grids: they seem to converge globally in space,
despite symmetry losses.
However, we are aware that multi-dimensional conservation laws can admit many entropy weak solutions,
and a precise study remains to be done.

\bigskip

Finally, let us summarize the main features 
of our new \emph{symmetric hyperbolic} system of conservation laws to model viscoelastic flows of (Upper-Convected) Maxwell fluids
with shear-waves propagating at finite-speed.

For the 2D gravity flows with a free surface detailed herein, 
the resulting generalization of Saint-Venant system has not only a theoretically interesting formulation.
But moreover, physically-reasonable numerical simulations have shown that our model indeed satisfies our initial goal.
This is encouraging from the conceptual viewpoint of physical principles
as well as for practical applications e.g. to transient geophysical flows.


However, to fully validate our new model and make it as useful as possible, a 
number of points need to be tackled.

\begin{itemize}
 \item Although it is satisfying that real (viscous) fluids seem well-modeled by symmetric-hyperbolic conservation laws,
 it is also well-known that the mathematical theory is far from complete at present.
 There remain challenging issues to precisely define physically-meaningful solutions at large times,
 and next simulate them numerically.
 
 Our FV approach certainly needs improving to that aim.

 
 \item The large-time asymptotic stability of sheared-flows needs to be investigated 
 theoretically and numerically, as boundary-value problems.
 
 \item In many practical applications, one would also want to use our new model with 3D flows ;
 adequate 
 choices of pressures should therefore be more carefully studied.
 Moreover, specific applications would also require one to modify the deviatoric stresses.
 It remains to investigate the various non-Newtonian possibilities in detail,
 with care as regards ``non-isothermal'' flows 
(recall our microscopic interpretation of the new state variable $\bA$ as a mean-field approximation of material distortion thermally-agitated).
\end{itemize}

\appendix

\section{
Riemann solvers based on flux-splitting 
}
\label{app:riemann}

The entropy-consistency condition \eqref{solversufficientcondition} 
can be achieved with simple 1D Riemann solvers which yield \emph{flux-splitting} FV schemes 
with a kinetic interpretation 
\cite{bouchut-2003}.
Here, we 
consider 1D Riemann solutions of the Lagrangian system \eqref{eq:SVUCM0detaillag}:
\beq
\label{eq:SVUCM0detaillag1D}
\begin{aligned}
& \partial_t H^{-1} - \partial_a( \Vcal_{a} \equiv U^j \sigma_{jk} F^k_b ) = 0 
\\
& \partial_t F^i_a - \partial_a U^i = 0
\\
& \partial_t F^i_b = 0
\\
& \partial_t U^i + \partial_a \left( \Pcal^i_a \equiv (gH^2/2 + GH^3 A_{cc}) \sigma_{ij} F^j_b - G F^i_\alpha A_{\alpha a} \right) = 0 
\\ 
& \partial_t A_{cc}^{1/4} = 0 
\end{aligned}
\eeq
that are consistent in the sense of \eqref{solversufficientcondition} with the following entropy inequality:
\beq
\label{eq:SVHElag1D}
\partial_t E + \partial_a\left(  \Pcal^i_a U^i \equiv U^i (gH^2/2 + GH^3 A_{cc}) \sigma_{ij} F^j_b - G U^i F^i_\alpha A_{\alpha a} \right) \le 0 
\eeq
when the flux-coefficients $A_{\alpha a}$ are given 
so the \emph{free-energy} functional:
\begin{equation}
\label{eq:internalenergylag1D} 
E = \frac12|\bU|^2 + \frac{g}2 H + \frac{G}2 \left( F^i_\alpha A_{\alpha\beta} F^i_\beta 
+ H^2 A_{cc} - \ln A_{cc} \right) \,,
\end{equation}
is actually a (strictly convex) mathematical entropy for 
\eqref{eq:SVUCM0detaillag1D}.

\subsection{Entropy-consistent solver in Lagrangian coordinates}
\label{app:lagrange}

The \emph{entropy-consistent} 
simple 1D Riemann solvers, satisfying \eqref{solversufficientcondition} with $\bn=\be_a$,
can be characterized 
in the flux-splitting case 
\cite[p.643]{bouchut-2003}. 
Recall: 
$E$ is a strictly convex entropy
for a system $\partial_t u + \partial_a F(u) = 0$ like~\eqref{eq:SVUCM0detaillag1D},
is equivalent to, 
the entropy variable $v \equiv \partial_u E$ 
symmetrizes the system 
and $F = \partial_v \psi$. Then \cite{bouchut-2003}:
\begin{lemma}
The FV schemes built with flux-splitting $F=F^++F^-$ 
for $\partial_t u + \partial_a F(u) = 0$ is entropy-consistent if
$F^+=\partial_v \psi^+$, $F^-=\partial_v \psi^-$ 
are built from two scalar functions $\psi^+,-\psi^-$ convex in $v$. 
\end{lemma}\noindent
Moreover, the entropy-consistency \eqref{solversufficientcondition} 
can be fully analyzed 
after kinetic interpretation 
like in \cite{bouchut-2003}. 
Then, considering 
$A_{cc}$ as flux-coefficient like $A_{\alpha\beta}$ 
for 
$u=(H^{-1},F^x_a,F^y_a,U^x,U^y)$ solution to \eqref{eq:SVUCM0detaillag1D}
here, let us split $F=\partial_v\psi$ where 
$$
v = \left( -\Pcal,\Gcal^x_a,\Gcal^y_a,U^x,U^y\right)
\quad
\psi = \Pcal\Vcal_a - \Gcal^x_a U^x - \Gcal^y_a U^y 
\equiv \Pcal^i_aU^i \,.
$$

Let us choose the following "convex" splitting:
$$
\psi^\epsilon = \frac{\epsilon}{4c} \sum_{i\in\{x,y\}} \left( \Pcal^i_a + c\epsilon U^i \right)^2 
\quad \epsilon\in\{+,-\}
$$
$$
F^\epsilon
= \frac\epsilon{2c} 
\begin{pmatrix}
-(\Pcal^x_a + c\epsilon U^x)F^y_b +(\Pcal^y_a + c\epsilon U^y)F^x_b \\
-(\Pcal^x_a + c\epsilon U^x) \\
-(\Pcal^y_a + c\epsilon U^y) \\
c\epsilon ( \Pcal^x_a + c\epsilon U^x ) \\
c\epsilon (-\Pcal^y_a + c\epsilon U^y )
\end{pmatrix}
$$
with a single parameter $c>0$ (to be fixed later, e.g. for entropy-consistency).
One can check that the entropy fluxes $G^\epsilon=v\cdot F^\epsilon - \psi^\epsilon$ associated with $F^\pm$, 
such that $\partial_u G^\epsilon = \partial_u E\cdot\partial_u F^\epsilon$,
recall $\partial_u\psi^\epsilon=(\partial_{uu}^2E)\partial_v\psi^\epsilon$, 
do ``dissipate" 
\cite{bouchut-2003}, i.e.
\beq
\label{eq1:dissip}
\epsilon \left( G^\epsilon(u_1) - G^\epsilon(u_2) -  (\partial_u E)|_{u_1} ( F^\epsilon(u_1) - F^\epsilon(u_2) ) \right) \le 0 \,.
\eeq
holds true for all 
$u_1,u_2$.
(Note indeed that \eqref{eq1:dissip} is equivalent to the convexity of $\epsilon\psi^\epsilon$, 
it can be checked by a direct computation here on noting $G^\epsilon = \psi^\epsilon$ 
and $\epsilon G^\epsilon = \epsilon \psi^\epsilon$ is a convex function of $v=\partial_u E$ such that $F^\epsilon=\partial_v \psi^\epsilon$).
%
Then, following \cite{bouchut-2003}, to fully check the entropy-consistency condition \eqref{solversufficientcondition} for 
the flux-splitting above as numerical flux
-- indeed a simple Riemann solver --, one can use its kinetic interpretation 
as a \emph{discretized} BGK model
with relaxation form\footnote{Where $\varepsilon$ 
assumes its usual meaning for relaxation systems \cite{bouchut-2003}: some time $\varepsilon\to0$ characteristic of the (numerical) relaxation
and infinitesimal in comparison with the time-step of the numerical scheme. It should not be mixed with $\epsilon\in\{+,-\}$.}:
\beq
\label{eq:SVUCM0detaillag1Drelaxed}
\begin{aligned}
& \partial_t H^{-1} - \partial_a \Vcal_a = 0
\\
& \partial_t \Vcal_a + \partial_a \Zcal_{aa} = 
({U^i \sigma_{ij} F^j_b - \Vcal_a})/\varepsilon 
\\
& \partial_t \Zcal_{aa} + c^2 \partial_a \Vcal_a = 
(\Pcal^i_a {\sigma_{ij}F^j_b - \Zcal_{aa}})/\varepsilon 
\\
& \partial_t F^i_b = 0
\\
& \partial_t F^i_a - \partial_a U^i = 0
\\
& \partial_t U^i + \partial_a \Pi^i_a 
= 0 
\\
& \partial_t \Pi^i_a + c^2 \partial_a U^i = 
({ \Pcal^i_a - \Pi^i_a })/\varepsilon
\end{aligned}
\eeq
also endowed with a diagonal $8\times8$ formulation: 
\beq
\label{eq:SVUCM0detaillag1Drelaxeddiag}
\begin{aligned}
& \partial_t \left( H^{-1} + \Zcal_{aa}/c^2 \right) 
= 
\frac{ H^{-1}  + \sigma_{ij}F^j_b\Pcal^i_a/c^2 - ( H^{-1} + \Zcal_{aa}/c^2)}\varepsilon
\\
& \partial_t \left( \Vcal_a \pm \Zcal_{aa}/c \right) \pm c\partial_a \left( \Vcal_a \pm \Zcal_{aa}/c \right) 
= 
\frac{ (U^i \pm \Pcal^i_a/c)\sigma_{ij}F^j_b - ( \Vcal_a \pm \Zcal_{aa}/c )}\varepsilon 
\\
& \partial_t \sigma_{jk} F^k_b = 0
\\
& \partial_t \left( \Pi^i_a/c^2 + F^i_a \right) 
= 
\frac{ \Pcal^i_a/c^2+F^i_a - \left( \Pi^i_a/c^2 + F^i_a \right) }\varepsilon
\\
& \partial_t \left( \Pi^i_a/c \pm U^i \right) \pm c \partial_a \left( \Pi^i_a/c \pm U^i \right) 
= 
\frac{ \Pcal^i_a/c\pm U^i - \left( \Pi^i_a/c \pm U^i \right) }\varepsilon
\end{aligned}
\eeq
that clearly shows the system \eqref{eq:SVUCM0detaillag1Drelaxed} has only linearly degenerate fields.

In particular, one can add 
$\partial_t A_{cc}=0$ without changing the structure of the Lagrangian system \eqref{eq:SVUCM0detaillag1Drelaxed},
as well as
$$
\partial_t ( |\bU|^2/2 + \hat e ) - \partial_a (\Pi^i_aU^i) = \left( E - (|\bU|^2/2 + \hat e)\right)/\varepsilon
$$
to check 
\eqref{solversufficientcondition} 
with $\tilde G_{\be_b}(u_l,u_r)=(\Pi^i_aU^i)|_{a/t=0}$ as in \cite{bouchut-2003}.
It amounts to add
\begin{multline}
\label{eq:extension}
\partial_t ( \hat e -  |\Pi^x_a|^2/2c^2 -  |\Pi^y_a|^2/2c^2 ) \\
= \frac{(E - |\bU|^2/2 -|\Pcal_a^x|^2/2c^2 - |\Pcal_a^y|^2/2c^2 ) - ( \hat e -  |\Pi^x_a|^2/2c^2 -  |\Pi^y_a|^2/2c^2 )}\varepsilon
\end{multline}
in 
\eqref{eq:SVUCM0detaillag1Drelaxeddiag}. Then, note that for \eqref{solversufficientcondition} to hold, it is sufficient that
$$
E(u) - |\bU|^2/2 -|\Pcal_a^x|^2/2c^2 - |\Pcal_a^y|^2/2c^2 \le \hat e -  |\Pi^x_a|^2/2c^2 -  |\Pi^y_a|^2/2c^2
$$
holds for all 
states $u$ in Riemann solutions of the extended 
system \eqref{eq:SVUCM0detaillag1Drelaxeddiag}--\eqref{eq:extension} i.e.
\beq
\label{eq:intermediateinequality}
E(u_1)-\frac{G^+(u_1) - G^-(u_1)}c \le E(u_2) + \frac{G^+(u_2) - G^-(u_2)}c 
\eeq
for all \emph{intermediate} state $u_1$ in the Riemann solution 
such that 
$(\hat e -  |\Pi^x_a|^2/2c^2 -  |\Pi^y_a|^2/2c^2) \equiv E(u_2) - \frac{G^+(u_2) - G^-(u_2)}c$
with $u_2$ the "outward" neighbour state (chosen in the direction opposite to the interface,
which is directly fixed by the boundary condition here in the 3-wave case).

Now, the split-form of 
the extended system \eqref{eq:SVUCM0detaillag1Drelaxeddiag}--\eqref{eq:extension} allows one to check \eqref{eq:intermediateinequality} 
independently for left and right (intermediate) states $u_1$
(with $u_2$ 
resp. the left and right boundary condition here),
when $A_{cc}$ is a variable solution to a contact disconitnuity wave,
and when the \emph{coefficient} $A_{\alpha\beta}$ is discontinuous in
\begin{multline*}
E - \frac{G^+ - G^-}c  
=  \frac{g}2 H +  \frac{G}2 \left( H^2 A_{cc} + F^i_\alpha A_{\alpha\beta} F^i_\beta  \right) - \ln A_{cc} 
- \frac1{2c^2} \left( |\Pcal_a^x|^2 + |\Pcal_a^y|^2  \right) 
\,.
\end{multline*}
%
Recalling \eqref{eq1:dissip} and $c(u_2-u_1)+F(u_2)-F(u_1)=0$ 
then \eqref{eq:intermediateinequality} rewrites
\begin{multline}
\label{eq:0wavedissipation}
E(u_1)-\frac{G^+(u_1) - G^-(u_1)}c - E(u_2) + \frac{G^+(u_2) - G^-(u_2)}c 
\\
- (\partial_u E)|_{u_1} \left( u_1-\frac{F^+(u_1) - F^-(u_1)}c - u_2 + \frac{F^+(u_2) - F^-(u_2)}c \right)
\le 0
\end{multline}
$$
F^+ - F^-  
= \frac1{c} 
\begin{pmatrix}
- \Pcal_a^i \sigma_{ij} F^j_b 
\\
- \Pcal_a^x \\
- \Pcal_a^y \\
c^2 U^x \\
c^2 U^y 
\end{pmatrix}
\quad
u - \frac{F^+ - F^-}c  
= 
\begin{pmatrix}
H^{-1} + \Pcal_a^i \sigma_{ij} F^j_b/c^2	 \\
F^x_a + \Pcal_a^x/c^2 \\
F^y_a + \Pcal_a^y/c^2 \\
0 \\
0 
\end{pmatrix}
$$
for boundary states $u_2=u_{l/r}$ resp. with neigbour intermediate state $u_1=u^*_{l/r}$
\beq
\label{eq:1paramsol}
\begin{aligned}
& H_l^* :=  \left( H_l^{-1} - \frac{\Vcal_l-\Zcal_l/c-\Vcal_r+\Zcal_r/c}{2c} \right)^{-1}
\\
& H_r^* :=  \left( H_r^{-1} - \frac{\Vcal_l+\Zcal_l/c-\Vcal_r-\Zcal_r/c}{2c} \right)^{-1}
\\
& ( \Vcal_a )_l^* = ( \Vcal_a )_r^* = \Vcal_a^* :=  \frac{\Vcal_l+\Zcal_l/c+\Vcal_r-\Zcal_r/c}{2} 
\\
& ( \Zcal_{aa} )_l^* = ( \Zcal_{aa} )_r^* = \Zcal_{aa}^* := \frac{c\Vcal_l+\Zcal_l-c\Vcal_r-\Zcal_r}{2} 
\\
& ( U^i )_l^* = ( U^i )_r^* = (U^i)^* :=  
\frac{(U^i)_l+( \Pi^i_a )_l/c+(U^i)_r-( \Pi^i_a )_r/c}{2}
\\
& ( \Pi^i_a )_l^* = ( \Pi^i_a )_r^* = (\Pi^i_a)^* := 
\frac{c(U^i)_l+( \Pi^i_a )_l-c(U^i)_r+( \Pi^i_a )_r}{2}
\\
\\
& (F^i_a)_l^* 
:=  
(F^i_a)_l - \frac{(U^i)_l-(\Pi^i_a)_l/c-(U^i)_r+(\Pi^i_a)_r/c}{2c}
\\
& (F^i_a)_r^* 
:=  
(F^i_a)_r - \frac{(U^i)_l+(\Pi^i_a)_l/c-(U^i)_r-(\Pi^i_a)_r/c}{2c}
\end{aligned}
\eeq
which can be checked e.g. using
\begin{multline}
\partial_{H_2^{-1}} \left(
E(u_2) - \sum_{\epsilon} \frac{\psi^\epsilon(u_2)}{\epsilon c} - (\partial_u E)|_{u_1} \left( u_2 - \sum_{\epsilon} \frac{F^\epsilon(u_2)}{\epsilon c} \right)
\right)
\\
= (\Pcal_1-\Pcal_2)\left( 1+ \frac{(\sigma_{ij}F^j_b)^2}{c^2} \partial_{H^{-1}}\Pcal|_2 \right)
\\
- GA_{aa}((F^i_a)_1-(F^i_a)_2) \left( \frac{\sigma_{ij}F^j_b}{c^2} \partial_{H^{-1}}\Pcal|_2 \right) \,,
\end{multline}
\begin{multline}
\frac1{GA_{aa}} \partial_{(F^i_a)_2} \left(
E(u_2) - \sum_{\epsilon} \frac{\psi^\epsilon(u_2)}{\epsilon c} - (\partial_u E)|_{u_1} \left( u_2 - \sum_{\epsilon} \frac{F^\epsilon(u_2)}{\epsilon c} \right)
\right)
\\
= ((F^i_a)_1-(F^i_a)_2) \left( \frac{GA_{aa}}{c^2} - 1\right) + (\Pcal_1-\Pcal_2) \frac{(\sigma_{ij}F^j_b)}{c^2} \,.
\end{multline}
When $u_1=u_2$, \eqref{eq:0wavedissipation} is satisfied,
and for $u_2=(-\Pcal,GA_{aa} F^x_a,GA_{aa} F^y_a,)_2\approx u_1$ 
\begin{multline}
d \left(
E(u_2) - \sum_{\epsilon,\delta} \frac{\psi^\epsilon(u_2)}{\epsilon c} 
- (\partial_u E)|_{u_1} \left( u_2 - \sum_{\epsilon} \frac{F^\epsilon(u_2)}{\epsilon c} \right)
\right)
\\
= - (\partial_{H^{-1}}\Pcal|_2) (\Pcal_1-\Pcal_2)^2 \left( (\partial_{H^{-1}}\Pcal|_2)^{-1} + \frac{(\sigma_{ij}F^j_b)^2}{c^2} \right)
\\
- (GA_{aa})^2 ((F^i_a)_1-(F^i_a)_2)^2 \left( \frac1{c^2} - (GA_{aa})^{-1}\right) 
\\
+ o\left( (\Pcal_1-\Pcal_2)^2,((F^i_a)_1-(F^i_a)_2)^2 \right)
\end{multline}
so on recalling $\partial_{H^{-1}}\Pcal = - (gH^3+3GH^4 A_{cc}) $, a 
sufficient condition
for entropy-consistency (of the 3-wave Riemann solver in Lagrange coordinates) reads:
\beq
\label{whitham3}
c^2 \ge \max\left( GA_{aa} , (gH^3+3GH^4 A_{cc})(F^j_b)^2 \right)
\eeq
on an \emph{admissible} neighbourhood of $u_1$ containing $u_2$ (such that $H>0$).
\begin{lemma}
\label{lem:entropylag3}
The 1D Riemann solver \eqref{eq:SVUCM0detaillag1Drelaxeddiag}--\eqref{eq:extension} for (the $\tilde q$ subsystem of) SVM in Lagrangian coordinates
is entropy-consistent in the sense of Prop.~\ref{prop:decay}, 
for the entropy $E$ given by \eqref{eq:internalenergylag1D} with numerical flux $\tilde G_{\be_b}=(\Pi_a^i U^i)|_{a/t=0}$,
provided \eqref{whitham3} holds on admissible neighbourhoods of the left/right states that resp. contain the left/right intermediate state.
\end{lemma}

Then, one can look for numerical values of $c$ satisfying \eqref{whitham3}, 
e.g. with $c$ 
solution to $\partial_t c=0$ as in \cite{bouchut-2004} to enforce $H_{l/r}^*>0$,
and it remains to map the solver above into Eulerian coordinates (see Section~\ref{sec:oned}).
Note however a \emph{consistency limitation} on the choice of 
$\be_a$ to approximate SVM Riemann solutions in material coordinates as above, with 3 waves only.
Indeed, the 1D solutions computed with \eqref{eq:SVUCM0detaillag1Drelaxed} 
are 
translation-invariant solutions to a 2D 
hyperbolic system of conservation laws in Lagrangian coordinates like for instance:
\beq
\label{eq:SVUCM0detaillag2Drelaxed}
\begin{aligned}
& \partial_t H^{-1} - \partial_\alpha \Vcal_\alpha = 0
\\
& \partial_t \Vcal_\alpha + \partial_\beta \Zcal_{\alpha\beta}
= 
({U^i\sigma_{\alpha\beta}\sigma_{ij}F^j_\beta - \Vcal_\alpha})/\varepsilon 
\\
& \partial_t \Zcal_{\alpha\beta} + c^2 \partial_\beta \Vcal_\alpha = 
({\Pcal^i_\alpha\sigma_{\alpha\beta}\sigma_{ij}F^j_\beta - \Zcal_{\alpha\beta}})/\varepsilon
\\
& \partial_t F^i_\alpha - \partial_\alpha U^i = 0
\\
& \partial_t U^i + \partial_\alpha \Pi^i_\alpha = 0 
\\
& \partial_t \Pi^i_\alpha + c^2 \partial_\alpha U^i = 
({ \Pcal^i_\alpha - \Pi^i_\alpha })/\varepsilon
\end{aligned}
\eeq
where $\alpha,\beta\in\{a,b\}$ refers to the axes of one particular Cartesian coordinate system.
This relaxation generalizes to
\beq
\label{eq:SVUCM0detaillag2Drelaxed2}
\begin{aligned}
& \partial_t H^{-1} - \partial_\alpha \Vcal_\alpha = 0
\\
& \partial_t \Vcal_\alpha + \partial_\beta \Zcal_{\alpha\beta}
= 
({U^i\sigma_{\alpha\beta}\sigma_{ij}F^j_\beta - \Vcal_\alpha})/\varepsilon 
\\
& \partial_t \Zcal_{\alpha\beta} + \tilde c^2_{\beta,\gamma} \partial_\gamma \Vcal_\alpha 
= 
({\Pcal^i_\alpha\sigma_{\alpha\beta}\sigma_{ij}F^j_\beta - \Zcal_{\alpha\beta}})/\varepsilon
\\
& \partial_t F^i_\alpha - \partial_\alpha U^i = 0
\\
& \partial_t U^i + \partial_\alpha \Pi^i_\alpha = 0 
\\
& \partial_t \Pi^i_\alpha + c^2_{\alpha,\gamma} \partial_\gamma U^i 
= 
({ \Pcal^i_\alpha - \Pi^i_\alpha })/\varepsilon
\end{aligned}
\eeq
with the \emph{same} symmetric-positive coefficient-matrices $\tilde c^2_{\alpha,\gamma} = c^2_{\alpha,\gamma}=c^2_{\gamma,\alpha}$ for $(\Pi^i_\alpha)_\alpha$, $(Z_{\beta\alpha})_\alpha$
in the wave equations resp. for $U^i$, $\Vcal_\beta$.
(The 1D Riemann solver \eqref{eq:SVUCM0detaillag1Drelaxed} coincides with 1D solutions
 of \eqref{eq:SVUCM0detaillag2Drelaxed2} along one principal directions). 
%
But 
it seems impossible to find a 2D system that admits the 1D Riemann solver \eqref{eq:SVUCM0detaillag1Drelaxed} as particular solution,
and that is a relaxation 
of the 2D Lagrangian SVM system \eqref{eq:SVUCM0detaillag} \emph{for all (smooth) solutions}.
Indeed, recalling $\partial_{H^{-1}}\Pcal=-(gH^3+3GH^4A_{cc})$,
$\Pi^i_\alpha$ in \eqref{eq:SVUCM0detaillag2Drelaxed2} above cannot consistently approximate \emph{any} exact solution $\Pcal^i_\alpha$ to
\begin{multline}
\label{eq:pressureeq}
\partial_t \Pcal^i_\alpha 
- \sigma_{ij}\sigma_{\alpha\beta} \Pcal \partial_\beta U^j
+ G A_{\alpha\beta} \partial_\beta U^i
\\
- (\partial_{H^{-1}}\Pcal) |\sigma_{ij}\sigma_{\alpha\beta}|^2 F^j_\beta 
\left( F^j_\beta  \partial_\alpha U^i - F^i_\beta  \partial_\alpha U^j + F^i_\alpha \partial_\beta U^j - F^j_\alpha \partial_\beta U^i \right) = 0
\end{multline}
insofar as the term $\partial_\alpha U^j$ in the last line of \eqref{eq:pressureeq} is always missing
(it vanishes though for 1D solutions along some direction $\beta$ such that $\sigma_{ij}F^j_\beta=0$ for all $i$).
So 
the 1D Riemann solver above with 3-wave cannot be consistent with any solution of the 2D SVM system,
even smooth. 

\subsection{Parameter initialization}
\label{app:parameter}

To achieve \eqref{whitham3} in a Riemann problem, 
one can require left/right initial values of the "relaxation parameter" $c$ as follows 
\begin{lemma}
\label{lem:c}
Defining for $\delta\in\{\parallel,\perp\}$ and $o=\{l,r\}$ in a Riemann problem
\beq
\label{eq:tildecpara}
\tilde c_o:=\sqrt{\max\{G(A_{aa})_o,(H F_b^\delta)^2_o(gH+3H^2A_{cc})_o\}}
\eeq
then conditions \eqref{whitham3} are ensured using as initial values (with $o\neq o'\{l,r\}$):
\beq
\label{eq:c} 
c_o = \tilde c_o + 2 H_o 
\left( \left[ (\Vcal_a)_l-(\Vcal_a)_r \right]_+ + \frac{\left[ (\Zcal_{aa})_{o'}-(\Zcal_{aa})_o \right]_+}{H_l \tilde c_l+H_r \tilde c_r}  \right) \,.
\eeq
\end{lemma}
\begin{proof}
The computation is straightforward and similar to Lemma 5.3 in \cite{bouchut-boyaval-2013}.
\end{proof}

\section{Hyperbolicity of SVUCM}
\label{app:hyperbolicity}

We investigate the hyperbolicity of SVUCM and variations with a non-zero slip-parameter, the so-called Johnson-Segalman models see e.g. \cite{boyaval-2018,boyaval-hal-01661269}. 

\begin{proposition} Among all 
Gordon-Schowalter derivatives with slip-parameter $\zeta\in[0,2]$, only 
$\zeta=0$ (i.e. the Upper-Convected 
case) 
ensures hyperbolicity of 
(\ref{eq:SVH}--\ref{eq:SVHUV}--\ref{eq:Sigmah}--\ref{eq:Sigmazz})
under 
strain-free constraints, namely: $H,B_{zz}>0$ and $\bB_h=\bB_h^T>0$. 
\end{proposition} 

The proof follows from rotation-invariance,
after computing the eigenvalues of the jacobian in a 1D projection of the system (\ref{eq:SVH}--\ref{eq:SVHUV}--\ref{eq:Sigmah}--\ref{eq:Sigmazz}) like
\beq
\label{eq:svucm1D}
\left\lbrace
\begin{aligned}
\partial_{t} \ro  + \partial_{x}( \ro \ux ) & = 0
\\
\partial_{t} \ux  + \ux\partial_{x} \ux &
+ g \partial_x \ro 
- G \left(  
  (\sx -\sz)/\ro \partial_x \ro
+ \partial_x (\sx -\sz)
  \right)
= 0
\\
\partial_{t}      \uy  
+ \ux\partial_{x} \uy 
&
- G \left( 
  \sxy/\ro \partial_x \ro
+ \partial_x\sxy
 \right)
= 0
\\
\partial_{t}      \sx  + \ux\partial_{x} \sx 
&
- \left( 
2 (1-\zeta)\sx \partial_x \ux 
-\zeta\sxy\partial_x\uy 
\right) 
= 0
\\
\partial_{t}      \sy  
+ \ux\partial_{x} \sy 
&
- 
\sxy 
(2-\zeta)\partial_x\uy 
= 0
\\
\partial_{t}      \sxy  
+ \ux\partial_{x} \sxy 
&
- 
\left( (1-\zeta/2)\sx-\zeta/2)\sy \right) \partial_x\uy 
- (1-\zeta) \sxy \partial_x\ux 
= 0
\\
\partial_{t}      \sz  + \ux\partial_{x} \sz 
&
+ 2(1-\zeta)\sz \partial_x\ux 
= 0
\end{aligned}
\right.
\eeq
similarly to the proof in \cite{EDWARDS1990411} for a similar system written in stress variables $\bSigma$ when $\zeta=0$
(though without vertical stress and strain components, which allow here mass preservation).
Denoting $\Delta = 2gh + G\left( 2(3-2\zeta)\sz+\zeta\sy-3\zeta\sx \right) = 2gh + 6G\sz + G\zeta\left( \sy-4\sz-3\sx \right)$, four eigenvalues read
$$
\ux\pm\frac12\sqrt{ 
 \Delta 
 + G\left( 4\sx-2\zeta(\sx+\sy) \right) 
\pm\sqrt{ 
\Delta^2
+ G^2(4\zeta\sxy)^2
}
}
$$
and are real if, and only if, the following strain-parametrized inequality holds
\begin{equation}
\label{eq:conditionreal} 
G^2(4\zeta\sxy)^2 \le 2G\Delta\left( 4\sx-2\zeta(\sx+\sy) \right) + G^2\left( 4\sx-2\zeta(\sx+\sy) \right)^2 \,.
\end{equation}
Unless $\zeta=0$, 
the quadratic condition \eqref{eq:conditionreal} on $\zeta$ is not clearly satisfied for all values $\sx,\sy,\sxy,\sz$ of the strain.
We therefore consider only 
(\ref{eq:SVH}--\ref{eq:SVHUV}--\ref{eq:Sigmah}--\ref{eq:Sigmazz}) when $\zeta=0$ (the SVUCM case),
where hyperbolicity is ensured with eigenvalues $u\pm\sqrt{gh+3G\sz+G\sx}$, $u\pm\sqrt{G\sx}$ and $u$
(with multiplicity 3) under the physcially-natural constraints $h\ge0,\sz\ge0,\sx\ge0$.

\providecommand{\bysame}{\leavevmode\hbox to3em{\hrulefill}\thinspace}
\providecommand{\MR}{\relax\ifhmode\unskip\space\fi MR }
\providecommand{\MRhref}[2]{%
  \href{http://www.ams.org/mathscinet-getitem?mr=#1}{#2}
}
\providecommand{\href}[2]{#2}

\end{document}